\documentclass{amsart}

\usepackage{amsmath}
\usepackage{hyperref}
\usepackage{enumerate}
\usepackage{amssymb}
\usepackage{graphicx}
\usepackage{MnSymbol}
\usepackage[all]{xy}
\usepackage{amsfonts}
\usepackage{eufrak}

\usepackage{amsthm}

\title{On projections of the tails of a power}

\author{Samuel M. Corson}
\address{School of Mathematics, University of Bristol, Fry Building, Woodland Road, Bristol, BS8 1UG, United Kingdom.}
\email{sammyc973@gmail.com}
\author{Saharon Shelah}
\address{Einstein Institute of Mathematics, The Hebrew University of Jerusalem, Jerusalem 91904 Israel}
\address{Department of Mathematics, Rutgers University, Piscataway, NJ 08854 USA}
\email{shelah@math.huji.ac.il}

\theoremstyle{definition}\newtheorem{theorem}{Theorem}

\theoremstyle{definition}

\newtheorem*{hugetheorem}{Main Theorem}
\numberwithin{theorem}{section}
\theoremstyle{definition}\newtheorem{corollary}[theorem]{Corollary}
\theoremstyle{definition}
\theoremstyle{definition}\newtheorem{definition}[theorem]{Definition}
\theoremstyle{definition}
\theoremstyle{definition}\newtheorem{example}{Example}
\theoremstyle{definition}\newtheorem{remark}[theorem]{Remark}
\theoremstyle{definition}
\theoremstyle{definition}\newtheorem{lemma}[theorem]{Lemma}
\theoremstyle{definition}
\theoremstyle{definition}
\theoremstyle{definition}
\theoremstyle{definition}
\theoremstyle{definition}\newtheorem{obs}[theorem]{Observation}
\theoremstyle{definition}\newtheorem{construction}[theorem]{Construction}
\theoremstyle{definition}\newtheorem{notation}[theorem]{Notation}

\newcommand{\Po}{\mathcal{P}}

\newcommand{\Card}{\operatorname{Card}}

\newcommand{\End}{\operatorname{End}}

\newcommand{\im}{\operatorname{im}}
\newcommand{\Atoms}{\textbf{Atoms}}
\newcommand{\B}{\mathfrak{B}}

\newcommand{\Av}{\textbf{Av}}
\newcommand{\cof}{\textbf{cof}}

\begin{document}

\keywords{endomorphism, idempotence, inaccessible cardinal}
\subjclass[2010]{20A15, 03E55, 03G05}
\thanks{The work of the first author was supported by the Additional Funding Programme for Mathematical
Sciences, delivered by EPSRC (EP/V521917/1) and the Heilbronn Institute
for Mathematical Research.  Also by the Basque Government Grant IT1483-22 and Spanish Government Grants PID2019-107444GA-I00 and PID2020-117281GB-I00.  The work of the second author was supported by JSF 1838/19: The Israel Science Foundation (2019-2023) and Rutgers 2018 DMS 1833363: NSF DMS Rutgers visitor program (PI S. Thomas) (2018-2022).  Paper number 1233 on Shelah's archive}

\begin{abstract}  Let $\kappa$ be an inaccessible cardinal, $\mathfrak{U}$ be a universal algebra, and $\sim$ be the equivalence relation on $\mathfrak{U}^{\kappa}$ of eventual equality.  From mild assumptions on $\kappa$ we give general constructions of $\mathcal{E} \in \End(\mathfrak{U}^{\kappa}/\sim)$ satisfying $\mathcal{E} \circ \mathcal{E} = \mathcal{E}$ which do not descend from $\Delta \in \End(\mathfrak{U}^{\kappa})$ having small strong supports.  As an application there exists an $\mathcal{E} \in \End(\mathbb{Z}^{\kappa}/\sim)$ which does not come from a $\Delta \in \End(\mathbb{Z}^{\kappa})$.
\end{abstract}

\maketitle

\begin{section}{Introduction}

Questions as to whether a mapping on a quotient can be lifted to the original structure are abundant in mathematics (in algebraic topology, group theory, etc.).  When the mapping cannot be (easily) lifted then it is of particular interest and intrinsic to the quotient structure.  Our focus is on the power of a universal algebra and its natural quotient- the set of ``tails''.

Let $\kappa$ be an infinite regular cardinal and $\mathfrak{U}$ a universal algebra.  We take each ordinal to be the set of ordinals strictly below it.  Write $\sigma \sim \tau$ if $\sigma, \tau \in \mathfrak{U}^{\kappa}$ are eventually equal (i.e. $|\{\alpha \in \kappa: \sigma(\alpha) \neq \tau(\alpha)\}| < \kappa$) and write $p_{\sim}: \mathfrak{U}^{\kappa} \rightarrow \mathfrak{U}^{\kappa}/\sim$ for the quotient homomorphism.  It is not difficult to produce a $\Delta \in \End(\mathfrak{U}^{\kappa})$ for which there is no $\mathcal{E} \in \End(\mathfrak{U}^{\kappa}/\sim)$ which makes the following diagram commute:

$$\xymatrix{\mathfrak{U}^{\kappa} \ar[r]^{\Delta} \ar[d]^{p_{\sim}}&
\mathfrak{U}^{\kappa} \ar[d]^{p_{\sim}} \\ \mathfrak{U}^{\kappa}/\sim
\ar[r]^{\mathcal{E}} & \mathfrak{U}^{\kappa}/\sim    }$$

\noindent This happens precisely when $\Delta$ makes essential use of early coordinates in defining cofinal coordinates.  For example the homomorphism $(\Delta\sigma)(\alpha) = \sigma(0)$ for all $\alpha \in \kappa$ does not descend to an $\mathcal{E} \in \End(\mathfrak{U}^{\kappa}/\sim)$ when the underlying set of $\mathfrak{U}$ has at least two elements.  By contrast it is generally less clear how to produce an $\mathcal{E}$ which is not descended from a $\Delta$.  In the cadence of calculus: Does there exist epsilon for which there is no delta?  We'll restrict our attention to the especially challenging situation where $\mathcal{E}$ is required to be a \emph{homomorphic projection} ($\mathcal{E} \circ \mathcal{E} = \mathcal{E}$).  The following classical example will motivate us.

\begin{example}\label{toQ}
Letting $\kappa = \aleph_0$, it is well-known that the abelian group $\mathbb{Z}^{\aleph_0}/\sim$ has a direct summand which is isomorphic to $\mathbb{Q}$.  Taking $\mathcal{E} \in  \End(\mathbb{Z}^{\aleph_0}/\sim)$ to be a homomorphic projection whose image is $\mathbb{Q} \leq \mathbb{Z}^{\aleph_0}/\sim$ we claim there is no $\Delta$ with $\mathcal{E} \circ p_{\sim} = p_{\sim} \circ \Delta$.

To see this, suppose such a $\Delta$ exists.  If $\Delta(\mathbb{Z}^{\aleph_0})$ is uncountable then $p_{\sim} \circ \Delta(\mathbb{Z}^{\aleph_0}) \simeq \mathbb{Q}$ is uncountable since $p_{\sim}$ has countable kernel.  Thus $\Delta(\mathbb{Z}^{\aleph_0})$ is countable, hence free abelian \cite[Theorem 3.8.2]{Fu}, and therefore finitely generated \cite[Lemmas 13.2.1, 13.2.3]{Fu}, and so $p_{\sim} \circ \Delta(\mathbb{Z}^{\aleph_0}) \simeq \mathbb{Q}$ is finitely generated, contradiction.
\end{example}

It is not obvious how to produce such an $\mathcal{E}$ when $\aleph_0$ is replaced with an uncountable regular cardinal $\kappa$, since $\mathbb{Z}^{\kappa}/\sim$ will then be $\aleph_1$-free abelian (hence cotorsion-free) \cite{Wald}, so in particular we cannot project to a nonzero divisible subgroup.  A large cardinal will allow us to determine such a homomorphism, and the map will be of the flavor of the next example.

\begin{example}\label{toZ}  Let $\kappa$ be a measurable cardinal and $\mathcal{U}$  an ultrafilter on $\kappa$ witnessing this.  Define $\Delta: \mathbb{Z}^{\kappa} \rightarrow \mathbb{Z}^{\kappa}$ by letting $\Delta\sigma$ be the constant sequence $(z_0)_{\alpha \in \kappa}$, where $\{\alpha \in \kappa: \sigma(\alpha) = z_0\} \in \mathcal{U}$.  As $\mathcal{U}$ is nontrivial and $\kappa$-complete, $\Delta$ descends to a homomorphic projection $\mathcal{E}$ with $\mathcal{E} \circ p_{\sim} = p_{\sim} \circ \Delta$.
\end{example}

Of course the $\mathcal{E}$ in Example \ref{toZ} descends from a $\Delta$, but the map is defined set-theoretically.  Moreover the construction is quite general as one can replace $\mathbb{Z}$ with $\mathfrak{U}$ having underlying set smaller than $\kappa$.

We'll provide two definitions toward stating the main theorem.  Letting $\mathfrak{U} = (\chi, \mathcal{S})$ we say a subset $Y \subseteq \kappa$ is a \emph{strong support} for homomorphism $\Theta: \mathfrak{U}^{\kappa} \rightarrow \mathfrak{M}$ if $\sigma \upharpoonright Y = \tau \upharpoonright Y$ implies $\Theta\sigma = \Theta\tau$ and for each $y \in Y$ and $\sigma \in \mathfrak{U}^{\kappa}$ there exists $x \in \chi$ such that $\tau \in \mathfrak{U}^{\kappa}$ given by

\[
\tau(\alpha) = \left\{
\begin{array}{ll}
\sigma(\alpha)
                                            & \text{if } \alpha \neq y , \\
x                                        & \text{if } \alpha = y
\end{array}
\right.
\]

\noindent has $\Theta \sigma \neq \Theta \tau$.  A homomorphism may or may not have a strong support.  If $\mathbb{F}$ is a field then we can take $\Theta: \mathbb{F}^{\kappa} \rightarrow \mathbb{F}$ to be any homomorphism which extends the homomorphism $\bigoplus_{\kappa}\mathbb{F} \rightarrow \mathbb{F}$ which takes the sum of the entries (using a vector space argument).  The strong support for this $\Theta$ is the entire set $\kappa$.  On the other hand if $\Theta$ is an extension of the map which takes the constant $(1_{\mathbb{F}})_{\alpha \in \kappa}$ to $1_{\mathbb{F}}$ and all elements of $\bigoplus_{\kappa}\mathbb{F}$ to $0_{\mathbb{F}}$, then $\Theta$ has no strong support.

We say a homomorphism $\Av: \mathfrak{U} \times \mathfrak{U} \rightarrow \mathfrak{U}$ is a \emph{changing average} if for all $x_0, x_1 \in \chi$ we have $\Av(x_0, x_1) = \Av(x_1, x_0)$ and $\Av(x_0, x_0) = x_0$ and there exist $x_2, x_3 \in \chi$ for which $x_2 \neq \Av(x_2, x_3) \neq x_3$.  In this paper we give two technical constructions which specialize to the following (see Remark \ref{Whatiskappa} Theorems \ref{firstconstruction} and \ref{combinginConstruction2}).

\begin{hugetheorem}\label{maintheorem}  Suppose that $\kappa$ is inaccessible with $\kappa^+ = 2^{\kappa}$.  If $\mathfrak{U} = (\chi, \mathcal{S})$ is a universal algebra with $2 \leq |\chi| < \kappa$ then there exists a homomorphic projection $\mathcal{E} \in \End(\mathfrak{U}^{\kappa}/\sim)$ such that

\begin{enumerate}

\item the image of $\mathcal{E}$ is isomorphic to $\mathfrak{U}^{\kappa}/\sim$; and

\item for any $\Delta \in \End(\mathfrak{U}^{\kappa})$ with $\mathcal{E} \circ p_{\sim} = p_{\sim} \circ \Delta$ we have 

\begin{center}

$|\{\alpha \in \kappa: \pi_{\alpha} \circ \Delta \text{ has no strong support of cardinality} \leq \lambda\}| = \kappa$

\end{center}

\noindent for each $\lambda < \kappa$.
\end{enumerate}

\noindent Additionally if $\mathfrak{U}$ has a changing average,

\begin{enumerate}
\item[(3)]  there exists $\sigma \in \mathfrak{U}^{\kappa}$ such that when $\mathcal{E}[\sigma]_{\sim} = [\tau]_{\sim}$ we get $$|\{\alpha \in \kappa: \sigma(\alpha) = \tau(\alpha)\}| < \kappa.$$

\end{enumerate}

\end{hugetheorem}

When there is no changing average the projection $\mathcal{E}$ essentially spreads some coordinates sideways in a complicated way.  When there is a changing average we obtain a projection which is more algebraically elaborate.  Of course, the theorem is most interesting for very specific algebras $\mathfrak{U}$.  For example, if $\mathfrak{U}$ is a slender abelian group $A$ (i.e. $A$ is torsion-free and includes no subgroups isomorphic to $\mathbb{Q}$, $\mathbb{Z}^{\aleph_0}$, or the p-adics $J_p$ with $p$ a prime) we know that every homomorphism from $A^{\kappa}$ to $A$ has a finite strong support when $\kappa$ is smaller than the least measurable cardinal \cite[Corollary 13.2.10]{Fu}.  The following is immediate.

\begin{corollary}\label{niceZ}  Let $A$ be a nontrivial slender abelian group.  If $\kappa$ is an inaccessible which is smaller than any measurable cardinal and $\kappa^+ = 2^{\kappa}$ then there exists a homomorphism projection $\mathcal{E} \in \End(A^{\kappa}/\sim)$ for which there is no $\Delta \in \End(A^{\kappa})$ such that $\mathcal{E} \circ p_{\sim} = p_{\sim} \circ \Delta$.  Moreover if $A = 2A$ then there is $\sigma \in A^{\kappa}$ for which $|\{\alpha \in \kappa: \sigma(\alpha) = \tau(\alpha)\}| < \kappa$ for $\mathcal{E}[\sigma]_{\sim} = [\tau]_{\sim}$. 
\end{corollary}

For the last sentence in the corollary we use the average function $\Av(x_0, x_1) = \frac{x_0 + x_1}{2}$.  We may let $A$ in Corollary \ref{niceZ} be $\mathbb{Z}$ or more generally a free abelian group of rank $< \kappa$, and for an application which involves the second sentence of the corollary we can take $A = \mathbb{Z}[\frac{1}{2}]$.  The hypothesis $\kappa^+ = 2^{\kappa}$ in the main theorem can be relaxed in a technical way (see Remark \ref{Whatiskappa}).  The assumptions on $\kappa$ used in our theorem ($\kappa$ is inaccessible and $\kappa^+ = 2^{\kappa}$) are quite mild compared to that used in Example \ref{toZ} ($\kappa$ is measurable); if $\mu$ is the least measurable cardinal then the set of inaccessible cardinals smaller than $\mu$ is of cardinality $\mu$ \cite[Lemma 10.21]{J}.

The proof of our theorem will involve modifications of the main arguments of \cite{ShSt}.  We give some preliminaries in Section \ref{prelim} and then give the first and second constructions in Sections \ref{Firstconstruction} and \ref{Secondconstruction} respectively.

\end{section}

\begin{section}{Some preliminaries}\label{prelim}

The reader may reference \cite{J} for most set-theoretic terminology.  We review some notational conventions below.

\begin{notation}  We will let 

\begin{itemize}

\item $\Card$ denote the class of cardinals;

\item $[X]^{\lambda} = \{Y: Y\subseteq X, |Y| = \lambda\}$;

\item $[X]^{<\lambda} = \{Y: Y \subseteq X, |Y| < \lambda\}$; 

\item $\Po(X)$ be the powerset of $X$;

\item $\text{}^YX$ be the set of functions from $Y$ to $X$;

\item $\cof(\alpha)$ denotes the cofinality of ordinal $\alpha$; and

\item $X \equiv_{\kappa} Y$ means that the symmetric difference $(X \cup Y) \setminus (X \cap Y)$ has cardinality less than $\kappa$.

\end{itemize}

\end{notation}

\begin{definition}  For functions $f, g \in \text{}^{\kappa}\Card$ we'll let $\mathfrak{d}_{f, g}$ denote the least cardinal of a family $\mathcal{D} \subseteq \prod_{\nu \in \kappa} [f(\nu)]^{g(\nu)}$ such that for each $F \in \prod_{\nu \in \kappa} f(\nu)$ there are $G \in \mathcal{D}$ and $\beta \in \kappa$ with $F(\nu) \in G(\nu)$ for $\nu \in \kappa \setminus \beta$.
\end{definition}

For the following, see \cite[Hypothesis 3.1]{ShSt2}.

\begin{definition}\label{dagger}  For a cardinal $\kappa$ we'll write $\dagger(\kappa, f, g)$ if $f, g \in   \text{}^{\kappa}(\kappa \cap \Card)$ are such that $f(\nu)$ and $g(\nu)$ infinite regular for all $\nu \in \kappa$ and 

\begin{enumerate}

\item $2^{g(\nu)} < f(\nu)$;

\item if $\nu \in \nu^*$ then $|\nu| \leq g(\nu) \leq g(\nu^*)$; and

\item $\mathfrak{d}_{2^{f}, g} = \kappa^+$.

\end{enumerate}

\end{definition}

\begin{remark}\label{Whatiskappa}  It is clear from (1), (2) of Definition \ref{dagger} that if $\dagger(\kappa, f, g)$ then $\nu < \kappa$ implies $2^{|\nu|} < \kappa$, and also $\lim_{\nu \rightarrow \kappa} g(\nu) = \kappa$.  We will also want $\kappa$ to be a regular cardinal for our arguments, so our attention will be strictly on $\kappa$ which is inaccessible.

When $\kappa$ is inaccessible with $\kappa^+ = 2^{\kappa}$ we can take $g$ and $f$ to be any functions satisfying conditions (1) and (2), take an enumeration $\{H_{\xi}\}_{\xi \in \kappa^+}$ of $\prod_{\nu \in \kappa} f(\nu)$, $e_{\xi}: \kappa \rightarrow \xi$ to be a bijection for each $\xi \in \kappa^+ \setminus \kappa$, and $G_{\xi}(\nu) = \{H_{e_{\xi(\eta)}}:\eta \in g(\nu)\}$ to witness $\dagger(\kappa, f, g)$.

It is also possible to have an inaccessible $\kappa$ and $f, g$ such that $\dagger(\kappa, f, g)$ and $\kappa^+ < 2^{\kappa}$  (see \cite[Corollary 3.2]{ShSt2}).

\end{remark}

The following appears as \cite[Lemma 3.1]{ShSt2}; we will provide the proof in extenso since that document has not been refereed and also for the sake of completeness.

\begin{lemma}\label{theZ} If $\kappa$ is inaccessible and $\dagger(\kappa, f, g)$ then there is a collection $\{Z_{\eta, \zeta}\}_{\eta \in \kappa^+, \zeta \in \kappa}$ satisfying the following:

\begin{enumerate}

\item  $Z_{\eta, \zeta} \subseteq \eta$;

\item  $|Z_{\eta, \zeta}| \leq g(\zeta)$;

\item  if $\zeta < \zeta^*$ then $Z_{\eta, \zeta} \subseteq Z_{\eta, \zeta^*}$;

\item  $\bigcup_{\zeta \in \kappa} Z_{\eta, \zeta} = \eta$;

\item  if $\eta \in \eta^*$ then there is $\beta \in \kappa$ such that $Z_{\eta^*, \zeta} \cap \eta = Z_{\eta, \zeta}$ for $\zeta \in \kappa \setminus \beta$; and

\item  if $\eta \in Z_{\eta^*, \zeta}$ then $Z_{\eta^*, \zeta} \cap \eta = Z_{\eta, \zeta}$.

\end{enumerate}

\end{lemma}

\begin{proof}  The sets $Z_{\eta, \zeta}$ are defined by induction on the first parameter $\eta \in \kappa^+$.  We let $Z_{0, \zeta} = \emptyset$ for all $\zeta \in \kappa$ (as required by condition (1)).  That (1) - (6) hold is easy to see.  Suppose that $Z_{\eta, \zeta}$ have been defined for all $\eta \in \xi$ and for all $\zeta \in \kappa$.  We'll treat three cases.

\noindent \textbf{Case One: $\xi = \xi^* + 1$.}  In this case we let $Z_{\xi, \zeta} = Z_{\xi^*, \zeta} \cup \{\xi^*\}$.  That (1) holds at this stage is quite clear, and (2) holds because $g(\zeta)$ is an infinite cardinal.  Condition (3) is clear, and (4) holds at $\xi$ because (4) holds at $\xi^*$.

Condition (5) is true by induction.  Let $\eta < \xi$ be given.  If $\eta < \xi^*$ then by induction we select $\beta \in \kappa$ such that $Z_{\xi^*, \zeta} \cap \eta = Z_{\eta, \zeta}$ for all $\zeta \in \kappa \setminus \beta$, and clearly $Z_{\xi, \zeta} \cap \eta = Z_{\eta, \zeta}$ for all $\zeta \in \kappa \setminus \beta$.  If $\eta = \xi^*$ then we can simply let $\beta = 0$.  Condition (6) holds by induction, looking at cases $\eta < \xi^*$ and $\eta = \xi$.

\noindent \textbf{Case Two: $\xi$ is a limit ordinal and $\cof(\xi) < \kappa$.}  Let $\{\gamma_{\nu}\}_{\nu < \cof(\xi)}$ be a strictly increasing sequence with $\sup_{\nu < \cof(\xi)}\gamma_{\nu} = \xi$.  As (5) holds below $\xi$ we select for each pair $\nu, \nu^*$ with $\nu < \nu^* < \cof(\xi)$ select $\beta_{\nu, \nu^*} \in \kappa$ such that $Z_{\gamma_{\nu^*}, \zeta} \cap \gamma_{\nu} = Z_{\gamma_{\nu}, \zeta}$ for all $\zeta \in \kappa \setminus \beta_{\nu, \nu^*}$.  As (4) holds below $\xi$, for each pair $\nu, \nu^*$ with $\nu < \nu^* < \cof(\xi)$ select $\epsilon_{\nu, \nu^*} \in \kappa$ such that $\gamma_{\nu} \in Z_{\gamma_{\nu^*}, \epsilon_{\nu, \nu^*}}$.  Let $\delta = \sup \{\beta_{\nu, \nu^*}\}_{\nu < \nu^* < \cof(\xi)} \cup\{\epsilon_{\nu, \nu^*}\}_{\nu < \nu^* < \cof(\xi)}$, and as $\cof(\xi) < \kappa$ we have $\delta \in \kappa$.  If $\alpha > \delta$ we have for all $\nu < \nu^* < \cof(\xi)$ that

$$
\begin{array}{ll}
\gamma_{\nu} & \in   Z_{\gamma_{\nu^* + 1}, \epsilon_{\nu, \nu^* + 1}} \cap \gamma_{\nu^*}\vspace*{2mm}\\

& \subseteq Z_{\gamma_{nu^* + 1}, \max(\epsilon_{\nu, \nu^* + 1}, \beta_{\nu^*, \nu^* + 1})} \cap \gamma_{\nu^*} \vspace*{2mm}\\
& \subseteq Z_{\gamma_{\nu^* + 1, \alpha}} \cap \gamma_{\nu^*} \vspace*{2mm}\\
& = Z_{\gamma_{\nu^*}}
\end{array}
$$

\noindent where the first two inclusions hold because condition (3) holds below $\xi$, and the equality on the last line holds by choice of $\beta_{\nu^*, \nu^* + 1}$.  Picking $\alpha \in \kappa \setminus (\max (\cof(\xi), \delta))$, we therefore know that for all $\zeta > \alpha$ and all $\nu < \nu^* < \cof(\xi)$ the equality $Z_{\gamma_{\nu^*}, \zeta} \cap \gamma_{\nu} = Z_{\gamma_{\nu}, \zeta}$ holds.  Let

\[
Z_{\xi, \zeta} = \left\{
\begin{array}{ll}
\emptyset
                                            & \text{if } \zeta \leq \alpha, \\
\bigcup_{\nu < \cof(\xi)}Z_{\gamma_{\nu}, \zeta}                                        & \text{if } \zeta  \geq  \alpha.
\end{array}
\right.
\]

\noindent That (1) holds is clear.  To see that condition (2) holds we point out that

\begin{center}
$|Z_{\xi, \zeta}| \leq \cof(\xi)g(\zeta) \leq |\alpha|g(\zeta) \leq |\zeta|g(\zeta) \leq g(\zeta)$
\end{center}

\noindent where the last inequality follows from condition (2) of $\dagger(\kappa, f, g)$.  For condition (3) we notice that given $\zeta < \zeta^* < \kappa$ we either have $\zeta \leq \alpha$, in which case $Z_{\xi, \zeta}  = \emptyset \subseteq Z_{\xi, \zeta^*}$, or $\alpha < \zeta$, and since $Z_{\gamma_{\nu}, \zeta} \subseteq Z_{\gamma_{\nu}, \zeta^*}$ holds for each $\nu < \cof(\xi)$ we get $Z_{\xi, \zeta} \subseteq Z_{\xi, \zeta^*}$.  Condition (4) is clear by induction.  To check condition (5) we let $\eta \in \xi$ be given and pick $\nu < \cof(\xi)$ such that $\eta < \gamma_{\nu}$.  As condition (5) holds at $\gamma_{\nu}$ we select $\beta \in \kappa \setminus \alpha$ such that $\zeta \in \kappa \setminus \beta$ implies $Z_{\gamma_{\nu}, \zeta} \cap \eta = Z_{\eta, \zeta}$.  Recall that by how $\alpha$ was selected we have $Z_{\gamma_{\nu^*}, \zeta} \cap \gamma_{\nu^{**}} = Z_{\gamma_{\nu^{**}}, \zeta}$ for all $\nu^{**} < \nu^* < \cof(\xi)$ and $\zeta \in \kappa \setminus \beta$.  Then for $\zeta \in \kappa \setminus \beta$ we have

$$
\begin{array}{ll}
Z_{\xi, \zeta} \cap \eta  & = (\bigcup_{\nu^* < \cof(\xi)}Z_{\gamma_{\nu^*}, \zeta}) \cap \eta \vspace*{2mm}\\

& = (\bigcup_{\nu^* < \cof(\xi)}Z_{\gamma_{\nu^*}, \zeta}) \cap \gamma_{\nu} \cap \eta\vspace*{2mm}\\
& = Z_{\gamma_{\nu}, \zeta} \cap \eta\vspace*{2mm}\\
& = Z_{\eta, \zeta}
\end{array}
$$

\noindent  For condition (6) we suppose that $\eta \in Z_{\xi, \zeta}$.  As $Z_{\xi, \eta} \neq \emptyset$ we have $\zeta > \alpha$, so $Z_{\xi, \zeta} = \bigcup_{\nu < \cof(\xi)}Z_{\gamma_{\nu}, \zeta}$.  Select $\nu < \cof(\xi)$ for which $\eta \in Z_{\gamma_{\nu}, \zeta}$.  Then $Z_{\gamma_{\nu}, \zeta} \cap \eta = Z_{\eta, \zeta}$ since condition (6) holds at $\gamma_{\nu}$.  As $\zeta \in \kappa \setminus \beta$ we know that $Z_{\gamma_{\nu^*}, \zeta} \cap \gamma_{\nu^{**}} = Z_{\gamma_{\nu^{**}}, \zeta}$ for all $\nu^{**} < \nu^* < \cof(\xi)$ and so

$$
\begin{array}{ll}
Z_{\xi, \zeta} \cap \eta  & = (\bigcup_{\nu^* < \cof(\xi)} Z_{\gamma_{\nu^*}, \zeta}) \cap \eta \vspace*{2mm}\\

& = (\bigcup_{\nu^* < \cof(\xi)}Z_{\gamma_{\nu^*}, \zeta}) \cap \gamma_{\nu} \cap \eta\vspace*{2mm}\\
& = Z_{\gamma_{\nu}, \zeta} \cap \eta\vspace*{2mm}\\
& = Z_{\eta, \zeta}
\end{array}
$$

\noindent and all conditions are satisfied.

\noindent \textbf{Case Three: $\cof(\xi) = \kappa$.}  We take $\{\gamma_{\nu}\}_{\nu\in \kappa}$ to be a strictly increasing sequence such that $\sup_{\nu \in \kappa} \gamma_{\nu} = \xi$.  As $\kappa$ is regular and condition (5) holds below $\xi$ we can take $\{\beta_{\theta}\}_{\theta \in \kappa}$ to be a strictly increasing sequence of ordinals in $\kappa$ such that $\beta_0 = 0$ and when $\nu < \theta$ we have for all $\zeta \geq \beta_{\theta}$ that $Z_{\gamma_{\theta}, \zeta} \cap \gamma_{\nu} = Z_{\gamma_{\nu}, \zeta}$.  Let $Z_{\xi, \zeta} = Z_{\gamma_{\theta}, \zeta}$ for $\beta_{\theta} \leq \zeta < \beta_{\theta + 1}$.  We check conditions (1) - (6).

To see why (1) and (2) hold, let $\zeta \in \kappa$ be given and take $\theta \in \kappa$ to the unique element such that $\beta_{\theta} \leq \zeta < \beta_{\theta + 1}$, and notice that $$Z_{\xi, \zeta} = Z_{\gamma_{\theta}, \zeta} \subseteq \gamma_{\theta} \subseteq \xi$$ and $$|Z_{\xi, \zeta}| = |Z_{\beta_{\theta}, \zeta}| \leq g(\zeta).$$  For (3) we let $\zeta <\zeta^* < \kappa$ and take $\beta_{\theta} \leq \zeta < \beta_{\theta + 1}$ and $\beta_{\theta^*} \leq \zeta^* < \beta_{\theta^* + 1}$.  If $\theta = \theta^*$ then $$Z_{\xi, \zeta} = Z_{\gamma_{\theta}, \zeta} \subseteq Z_{\gamma_{\theta^*}, \zeta^*} = Z_{\xi, \zeta^*}$$ since condition (3) holds below $\xi$, and if $\theta < \theta^*$ then $$Z_{\xi, \zeta} = Z_{\gamma_{\theta}, \zeta} \subseteq Z_{\gamma_{\theta}, \zeta^*} = Z_{\gamma_{\theta^*}, \zeta^*} \cap \gamma_{\theta} \subseteq Z_{\gamma_{\theta^*}, \zeta^*} = Z_{\xi, \zeta^*}.$$  For condition (4), we notice first that $\bigcup_{\zeta \in \kappa} Z_{\xi, \kappa} \subseteq \xi$ as condition (1) holds.  For the reverse inclusion we let $\alpha \in \xi$ be given.  Select $\nu \in \kappa$ such that $\alpha \in \gamma_{\nu}$, and as condition (4) holds at $\gamma_{\nu}$ we select $\zeta^* \in \kappa$ such that $\alpha \in Z_{\gamma_{\nu}, \zeta^*}$.  Select $\nu < \theta \in \kappa$ so that there exists $\zeta^{**} \in \kappa$ with $\zeta^{**} > \zeta^*$ and $\beta_{\theta} \leq \zeta^{**} < \beta_{\theta + 1}$.  Then

$$\alpha \in Z_{\gamma_{\nu}, \zeta^*} \subseteq Z_{\gamma_{\nu}, \zeta^{**}}= Z_{\gamma_{\theta}, \zeta^{**}} \cap \gamma_{\nu} \subseteq Z_{\gamma_{\theta}, \zeta^{**}} = Z_{\xi, \zeta^{**}}\subseteq \bigcup_{\zeta \in \kappa} Z_{\xi, \zeta}$$

For condition (5) we let $\eta \in \xi$.  Select $\theta \in \kappa$ large enough that $\eta \in \gamma_{\theta}$.  As condition (5) holds at $\gamma_{\theta}$ we select $\beta \in \kappa$ such that $\zeta \in \kappa \setminus \beta$ implies $Z_{\gamma_{\theta}, \zeta} \cap \eta = Z_{\eta, \zeta}$.  Fixing $\zeta \in \kappa \setminus \max(\beta, \beta_{\theta})$ we select $\theta^* \in \kappa$ with $\beta_{\theta^*} \leq \zeta < \beta_{\theta^* + 1}$.  Then $\theta \leq \theta^*$ and $Z_{\xi, \zeta} = Z_{\gamma_{\theta^*}, \zeta}$, so in particular

$$Z_{\eta, \zeta} = Z_{\gamma_{\theta}, \zeta} \cap \eta = Z_{\gamma_{\theta^*}, \zeta} \cap \gamma_{\theta} \cap \eta = Z_{\xi, \zeta} \cap \eta.$$

To see condition (6), if $\eta \in Z_{\xi, \zeta} = Z_{\gamma_{\theta}, \zeta}$ then $$Z_{\xi, \zeta} \cap \eta = Z_{\gamma_{\theta}, \zeta} \cap \eta = Z_{\eta, \zeta}$$ since condition (6) holds below $\xi$, and all conditions hold.

\end{proof}

\begin{definition}  If $\B$ is an atomic Boolean algebra we say that a function $\Phi: \B \rightarrow \B$ a $\jmath$-\emph{mapping} if

\begin{enumerate}[(a)]

\item $\Phi \circ \Phi = \Phi$;

\item $\Phi(\Atoms(\B)) \subseteq \Atoms(\B)$;

\item $\Phi \upharpoonright (\Atoms(\B) \setminus \Phi(\Atoms(\B)))$ is a bijection with $\Phi(\Atoms(\B))$; and

\item $\Phi(\bigcup_{A \in X}A) = \bigcup_{A \in X}\Phi(A)$  for any $X \subseteq \Atoms(\B)$

\end{enumerate}

\noindent where $\Atoms(\B)$ denotes the set of atoms of $\B$.

\end{definition}

Now we give the crucial lemma from which the constructions will follow, which is a modification of the ideas in \cite{ShSt}.

\begin{lemma}\label{sequences}  Assume $\kappa$ is inaccessible and $\dagger(\kappa, f, g)$ and let $\{I_{\nu}\}_{\nu \in \kappa}$ be a partition of $\kappa$ such that $|I_{\nu}|= f(\nu)$.  Let $p: \kappa \rightarrow \kappa$ be defined by $\alpha \in I_{p(\alpha)}$.  For each $\nu \in \kappa$ write $I_{\nu} = I_{0, \nu} \sqcup I_{1, \nu}$, where $|I_{0, \nu}| = |I_{1, \nu}| = |I_{\nu}|$ and let $\psi_{\nu}: I_{\nu} \rightarrow I_{1, \nu}$ be a function such that $\psi_{\nu} \upharpoonright I_{1, \nu}$ is identity and $\psi_{\nu} \upharpoonright I_{0, \nu}$ is a bijection with $I_{1, \nu}$.

Then there exist $\{\B_{\xi, \nu}\}_{\xi\in \kappa^+, \nu \in \kappa}$, and $\{\Phi_{\xi, \nu}\}_{\xi \in \kappa^+, \nu \in \kappa}$ such that

\begin{enumerate}

\item $\B_{\xi, \nu}$ is an atomic Boolean subalgebra of $\Po(I_{\nu})$ and $\Phi_{\xi, \nu}: \B_{\xi, \nu} \rightarrow \B_{\xi, \nu}$ is a $\jmath$-mapping for each $\xi \in \kappa^+$ and $\nu \in \kappa$;

\item if $\xi \in \eta \in \kappa^+$ then there is a $\beta \in \kappa$ such that $\B_{\xi, \nu} \subseteq \B_{\eta, \nu}$ and $\Phi_{\xi, \nu} \subseteq \Phi_{\eta, \nu}$ for all $\nu \in \kappa \setminus \beta$;

\item for $A \in \Po(\kappa)$ there are $\xi \in \kappa^+$ and $\beta \in \kappa$ with $A \cap I_{\nu} \in \B_{\xi, \nu}$ for all $\nu \in \kappa \setminus \beta$.

\item for any collection $\{F_{\alpha}\}_{\alpha \in \kappa}$ of sets  such that $|F_{\alpha}| < f(p(\alpha))$ and $F_{\alpha} \subseteq I_{0, p(\alpha)}$ there are $\xi \in \kappa^+$, $\beta \in \kappa$ and sequence $\{\alpha_{\nu}\}_{\nu \in \kappa \setminus \beta}$ for which $F_{\alpha_{\nu}} \in \B_{\xi, \nu}$ and $\alpha_{\nu} \in I_{1, \nu} \setminus \Phi_{\xi, \nu}(F_{\alpha_{\nu}})$;

\end{enumerate}

\end{lemma}

\begin{proof}  Let $\{G_{\xi}\}_{\xi \in \kappa^+}$ witness $\mathfrak{d}_{2^{f}, g} = \kappa^+$.  For each $\nu \in \kappa$ let $\{E_{\theta, \nu}\}_{\theta \in 2^{f(\nu)}}$ enumerate $\Po(I_{\nu})$ and let $\{L_{\theta, \nu}\}_{\theta \in 2^{f(\nu)}}$ enumerate all functions from $I_{1, \nu}$ to $[I_{0, \nu}]^{< f(\nu)}$.

We inductively define the $\{\B_{\xi, \nu}\}_{\xi\in \kappa^+, \nu \in \kappa}$, and $\{\Phi_{\xi, \nu}\}_{\xi \in \kappa^+, \nu \in \kappa}$.  Let $\{Z_{\eta, \zeta}\}_{\eta \in \kappa^+, \zeta \in \kappa}$ be the collection given by Lemma \ref{theZ}.  For $\nu \in \kappa$ we let $\B_{\Psi, \nu} = \{\emptyset, I_{\nu}, I_{0, \nu}, I_{1, \nu}\}$ and $\Psi_{\nu}: \B_{\Psi, \nu} \rightarrow \B_{\Psi, \nu}$ be the $\jmath$-mapping induced by $\psi_{\nu}$.

For $\nu \in \kappa$ let $\B_{0, \nu} = \B_{\Psi, \nu}$ and $\phi_{0, \nu}$ be the empty mapping.  Suppose that $\B_{\xi, \nu}$ and $\phi_{\xi, \nu}$ have been defined for all $\zeta \in \eta$ and $\nu \in \kappa$ satisfying the following conditions:

\begin{enumerate}

\item for $\xi \in \eta$ and $\nu \in \kappa$

\begin{enumerate}[(a)]

\item $\phi_{\xi, \nu}: D_{\xi, \nu} \rightarrow D_{\xi, \nu}$ is such that 

\begin{enumerate}[(i)]

\item $\phi_{\xi, \nu} \upharpoonright (D_{\xi, \nu} \cap I_{0, \nu})$ is a bijection with $D_{\xi, \nu} \cap I_{1, \nu}$,

\item $\phi_{\xi, \nu} \upharpoonright (D_{\xi, \nu} \cap I_{1, \nu})$ is identity, and

\item for $x \in D_{\xi, \nu} \cap I_{0, \nu}$ we have $\phi_{\xi, \nu}\circ (\psi_{\nu}\upharpoonright I_{0, \nu})^{-1}\circ\phi_{\xi, \nu}(x) = \psi_{\nu}(x)$;

\end{enumerate}

\item $|D_{\xi, \nu}| < f(\nu)$;

\item $\B_{\xi, \nu}$ is an atomic Boolean subalgebra of $\Po(I_{\nu})$ such that

\begin{enumerate}[(i)]

\item $\B_{\psi, \nu} \subseteq \B_{\xi, \nu}$, and

\item for all $X \in \B_{\xi, \nu}$ we have $\psi_{\nu}(X), \psi_{\nu}^{-1}(X) \in \B_{\xi, \nu}$;

\end{enumerate}

\item $D_{\xi, \nu} \in \B_{\xi, \nu}$;

\item $|\Atoms(\B_{\xi, \nu})| \leq 2^{g(\nu)}$;

\end{enumerate}

\item if $\overline{\eta} \in \eta$ and $\xi \in Z_{\overline{\eta}, \nu}$ then

\begin{enumerate}[(a)]

\item $\B_{\xi, \nu} \subseteq \B_{\overline{\eta}, \nu}$;

\item $\phi_{\xi, \nu} \subseteq \phi_{\overline{\eta}, \nu}$;

\item $\Psi_{\nu} \subseteq \Phi_{\xi, \nu} \subseteq \Phi_{\overline{\eta}, \nu}$ where $\Phi_{\xi, \nu}: \B_{\xi, \nu} \rightarrow \B_{\xi, \nu}$ is a $\jmath$-mapping given by $\Phi_{\xi, \nu}(A) = \phi_{\xi, \nu}(A \cap D_{\xi, \nu}) \cup \psi_{\nu}(A \setminus D_{\xi, \nu})$.

\end{enumerate}

\end{enumerate}

By condition (6) of Lemma \ref{theZ} and induction hypothesis 2(a) we see that if $\xi \in \xi^*$ and $\xi, \xi^* \in Z_{\eta, \nu}$ then $\B_{\xi, \nu}\subseteq \B_{\xi^*, \nu}$.  Moreover, because $|Z_{\eta, \nu}| \leq g(\nu)$ it follows from induction hypothesis 1(e) that letting

\begin{center}

$\mathcal{J}_{\eta, \nu} = \{\bigcap_{\xi \in Z_{\eta, \nu}}j(\xi)\mid j \in \prod_{\xi \in Z_{\eta, nu}}\Atoms(\B_{\xi, \nu})\}$

\end{center}

\noindent we have $|\mathcal{J}_{\eta, \nu}| \leq (2^{g(\nu)})^{g(\nu)} = 2^{g(\nu)}$ for each $\nu$.  Furthermore, letting $\mathcal{K}_{\eta, \nu}$ be the partition of $I_{\nu}$ generated by $\{E_{\theta, \nu}\}_{\theta \in G_{\eta}(\nu)}$ we see also $|\mathcal{K}_{\eta, \nu}| \leq 2^{|G_{\eta}(\nu)|} = 2^{g(\nu)}$.

We point out that the union $\bigcup_{\xi \in Z_{\eta, \nu}}\phi_{\xi, \nu}$ is a function, for if $\xi, \xi^* \in Z_{\eta, \nu}$ and $\xi \in \xi^*$ then $Z_{\eta, \nu} \cap \xi^* = Z_{\xi^*, \nu}$ (by condition (6) of Lemma \ref{theZ}), and by induction hypothesis 2(b) we have $\phi_{\xi, \nu} \subseteq \phi_{\xi^*, \nu}$.  By induction hypothesis 1(a) we see that $\bigcup_{\xi \in Z_{\eta, \nu}}\phi_{\xi, \nu}$ has domain $\bigcup_{\xi \in Z_{\eta, \nu}} D_{\xi, \nu}$ and moreover

\begin{enumerate}[(i)]

\item $(\bigcup_{\xi \in Z_{\eta, \nu}}\phi_{\xi, \nu}) \upharpoonright (\bigcup_{\xi \in Z_{\eta, \nu}} D_{\xi, \nu} \cap I_{0, \nu})$ is a bijection with $\bigcup_{\xi \in Z_{\eta, \nu}} D_{\xi, \nu} \cap I_{1, \nu}$;

\item $(\bigcup_{\xi \in Z_{\eta, \nu}}\phi_{\xi, \nu}) \upharpoonright (\bigcup_{\xi \in Z_{\eta, \nu}} D_{\xi, \nu} \cap I_{1, \nu})$ is identity; and

\item for $x \in \bigcup_{\xi \in Z_{\eta, \nu}} D_{\xi, \nu} \cap I_{0, \nu}$ we have $(\bigcup_{\xi \in Z_{\eta, \nu}}\phi_{\xi, \nu}) \circ (\psi_{\nu} \upharpoonright I_{0, \nu})^{-1} \circ (\bigcup_{\xi \in Z_{\eta, \nu}}\phi_{\xi, \nu})(x) = \psi_{\nu}(x)$.

\end{enumerate}

\noindent We let

\begin{center}

$\underline{D_{\eta, \nu}} = \bigcup_{\xi \in Z_{\eta, \nu}} D_{\xi, \nu}$

\end{center}

\noindent and

\begin{center}

$\underline{\phi_{\eta, \nu}} = (\psi_{\nu}\upharpoonright I_{\nu} \setminus \underline{D_{\eta, \nu}}) \cup \bigcup_{\xi \in Z_{\eta, \nu}}\phi_{\xi, \nu}$

\end{center}

\noindent Also, since $f(\nu)$ is regular and $|D_{\xi, \nu}| < f(\nu)$ for all $\xi$ we have $|\underline{D_{\eta, \nu}}| < f(\nu)$.  Clearly for $x \in I_{0, \nu}$ we have $\underline{\phi_{\eta, \nu}} \circ (\psi_{\nu} \upharpoonright I_{0, \nu})^{-1}\circ \underline{\phi_{\eta, \nu}}(x) = \psi_{\nu}(x)$.  We let $\underline{\mathcal{A}}_{\eta, \nu}$ be the partition of $I_{\nu}$ generated by $\mathcal{J}_{\eta, \nu} \cup \mathcal{K}_{\eta, \nu} \cup \{I_{0, \nu}, I_{1, \nu}\}$, so $|\underline{\mathcal{A}}_{\eta, \nu}| \leq 2^{g(\nu)}$.  As $\underline{\mathcal{A}}_{\eta, \nu}$ is a partition of $I_{\nu}$, for each $x \in I_{0, \nu}$ there are unique $W_{0, x}, W_{1, x}, W_{2, x}, W_{3, x} \in \underline{\mathcal{A}}_{\eta, \nu}$ such that

\begin{itemize}

\item $x \in W_{0, x}$;

\item $\psi_{\nu}(x) \in W_{1, x}$;

\item $(\psi_{\nu}\upharpoonright I_{0, \nu})^{-1} \circ \underline{\phi_{\eta, \nu}}(x) \in W_{2, x}$; and

\item $\underline{\phi_{\eta, \nu}}(x) \in W_{3, x}$.

\end{itemize}

\noindent Thus for $x \in I_{0, \nu}$ we let $$A_x = W_{0, x} \cap \psi_{\nu}^{-1}(W_{1, x}) \cap \underline{\phi_{\eta, \nu}}^{-1}(\psi_{\nu}(W_{2, x})) \cap \underline{\phi_{\eta, \nu}}^{-1}(W_{3, x}) \cap I_{0, \nu}$$

\noindent and define $\mathcal{H}_0 \subseteq \Po(I_{0, \nu})$ by $\mathcal{H}_0 = \{A_x\}_{x \in I_{0, \nu}}$.  It is clear that $\mathcal{H}_0$ is a partition of $I_{0, \nu}$.  Let $\mathcal{H}_1 = \{\psi_{\nu}(A_x)\}_{x \in I_{0, \nu}}$ and $\mathcal{A}_{\eta, \nu} = \mathcal{H}_0 \cup \mathcal{H}_1$, and it is clear that $\mathcal{A}_{\eta, \nu}$ is a partition of $I_{\nu}$ and $|\mathcal{A}_{\eta, \nu}| \leq 2^{g(\nu)}$.

We note that for any $x \in I_{0, \nu}$ we have $\psi_{\nu}(A_x) \in \mathcal{H}_1$ (by definition of $\mathcal{H}_1$) but also $\underline{\phi_{\eta, \nu}}(A_x) \in \mathcal{H}_1$.  To see why this latter claim is true we let $x \in I_{0, \nu}$ be given and let $y = (\psi_{\nu} \upharpoonright I_{0, \nu})^{-1} \circ \underline{\phi_{\eta, \nu}}(x)$.  We have $(\psi_{\nu} \upharpoonright I_{0, \nu})^{-1} \circ \underline{\phi_{\eta, \nu}}(y) = x$.  Then $W_{0, x} = W_{2, y}$, $W_{1, x} = W_{3, y}$, $W_{2, x} = W_{0, y}$, $W_{3, x} = W_{1, y}$.  It is readily seen that

$$
\begin{array}{ll}
\underline{\phi_{\eta, \nu}}(A_x) & = \underline{\phi_{\eta, \nu}}(W_{0, x} \cap \psi_{\nu}^{-1}(W_{1, x}) \cap \underline{\phi_{\eta, \nu}}^{-1}(\psi_{\nu}(W_{2, x})) \cap \underline{\phi_{\eta, \nu}}^{-1}(W_{3, x}) \cap I_{0, \nu})\vspace*{2mm}\\
 & = \underline{\phi_{\eta, \nu}}(W_{0, x}) \cap \underline{\phi_{\eta, \nu}}\circ (\psi_{\nu}\upharpoonright I_{0, \nu})^{-1}(W_{1, x})\vspace*{2mm}\\
& \text{   }\cap \underline{\phi_{\eta, \nu}}\circ (\underline{\phi_{\eta, \nu}}\upharpoonright I_{0, \nu})^{-1}(\psi_{\nu}(W_{2, x})) \cap \underline{\phi_{\eta, \nu}} \circ (\underline{\phi_{\eta, \nu}}\upharpoonright I_{0, \nu})^{-1}(W_{3, x})\vspace*{2mm}\\
& = \underline{\phi_{\eta, \nu}}(W_{0, x}) \cap \underline{\phi_{\eta, \nu}}\circ (\psi_{\nu}\upharpoonright I_{0, \nu})^{-1}(W_{1, x}) \cap \psi_{\nu}(W_{2, x}) \cap W_{3, x}\vspace*{2mm}\\
& = \underline{\phi_{\eta, \nu}}(W_{2, y}) \cap \underline{\phi_{\eta, \nu}}\circ (\psi_{\nu}\upharpoonright I_{0, \nu})^{-1}(W_{3, y}) \cap \psi_{\nu}(W_{0, y}) \cap W_{1, y}\vspace*{2mm}\\
& = \psi_{\nu} \circ (\psi_{\nu}\upharpoonright I_{0, \nu})^{-1}(\underline{\phi_{\eta, \nu}}(W_{2, y}) \cap \underline{\phi_{\eta, \nu}}\circ (\psi_{\nu}\upharpoonright I_{0, \nu})^{-1}(W_{3, y})\vspace*{2mm}\\
& \text{   }\cap \psi_{\nu}(W_{0, y}) \cap W_{1, y})\vspace*{2mm}\\
& = \psi_{\nu}((\psi_{\nu}\upharpoonright I_{0, \nu})^{-1} \circ \underline{\phi_{\eta, \nu}}(W_{2, y}) \cap (\psi_{\nu}\upharpoonright I_{0, \nu})^{-1} \circ \underline{\phi_{\eta, \nu}}\circ (\psi_{\nu}\upharpoonright I_{0, \nu})^{-1}(W_{3, y})\vspace*{2mm}\\
& \text{   } \cap (\psi_{\nu}\upharpoonright I_{0, \nu})^{-1} \circ \psi_{\nu}(W_{0, y}) \cap (\psi_{\nu}\upharpoonright I_{0, \nu})^{-1}(W_{1, y}))\vspace*{2mm}\\
& = \psi_{\nu}((\psi_{\nu}\upharpoonright I_{0, \nu})^{-1} \circ \underline{\phi_{\eta, \nu}}(W_{2, y}) \cap (\underline{\phi_{\eta, \nu}} \upharpoonright I_{0, \nu})^{-1}(W_{3, y}) \cap W_{0, y}\vspace*{2mm}\\
& \text{   }\cap (\psi_{\nu} \upharpoonright I_{0, \nu})^{-1}(W_{1, y}))\vspace*{2mm}\\
& = \psi_{\nu}(\psi_{\nu}^{-1} \circ \underline{\phi_{\eta, \nu}}(W_{2, y}) \cap \underline{\phi_{\eta, \nu}}^{-1}(W_{3, y}) \cap W_{0, y} \cap \psi_{\nu}^{-1}(W_{1, y}) \cap I_{0, \nu})\vspace*{2mm}\\
& = \psi_{\nu}(\underline{\phi_{\eta, \nu}}^{-1}(\psi_{\nu}(W_{2, y})) \cap \underline{\phi_{\eta, \nu}}^{-1}(W_{3, y}) \cap W_{0, y} \cap \psi_{\nu}^{-1}(W_{1, y}) \cap I_{0, \nu})\vspace*{2mm}\\
& = \psi_{\nu}(A_y) \in \mathcal{H}_1

\end{array}
$$

\noindent It is similarly seen that if $A \subseteq I_{1, \nu}$ and $A \in \mathcal{A}_{\eta, \nu}$ then both $(\psi_{\nu}\upharpoonright I_{0, \nu})^{-1}(A) \in \mathcal{A}_{\eta, \nu}$ and $(\underline{\phi_{\eta, \nu}})^{-1}(A) \in \mathcal{A}_{\eta, \nu}$.  As $f(\nu)$ is regular and $|\mathcal{A}_{\eta, \nu}| \leq 2^{g(\nu)} < f(\nu)$ we have some $A_{\eta, \nu} \in \mathcal{A}_{\eta, \nu}$ with $|A_{\eta, \nu}| = f(\nu)$, and moreover we can take $A_{\eta, \nu} \subseteq I_{0, \nu}$ by replacing $A_{\eta, \nu}$ with $(\psi_{\nu}\upharpoonright I_{0, \nu})^{-1}(A_{\eta, \nu})$ if necessary.

We let $B_{\nu}^1 \in [\psi_{\nu}(A_{\eta, \nu})]^{|G_{\eta}(\nu)|}$.  Select $B_{\nu}^0 \in [A_{\eta, \nu}]^{|G_{\eta}(\nu)|}$ such that $(\bigcup_{x \in B_{\nu}^1}L_{\theta, \nu}(x)) \cap B_{\nu}^0 = \emptyset$ for all $\theta \in  G_{\eta}(\nu)$ and $\psi_{\nu}(B_{\nu}^0) \cap B_{\nu}^1 = \emptyset$.  Let $B_{\nu}^2 = (\psi_{\nu}\upharpoonright I_{0, \nu})^{-1}(B_{\nu}^1))$ and $B_{\nu}^3 = \psi_{\nu}(B_{\nu}^0)$.  Let $\phi: B_{\nu}^0 \cup B_{\nu}^1 \cup B_{\nu}^2 \cup B_{\nu}^3 \rightarrow B_{\nu}^1 \cup B_{\nu}^3$ be a function such that

\begin{itemize}

\item $\phi \upharpoonright (B_{\nu}^1 \cup B_{\nu}^3)$ is the identity map;

\item $\phi \upharpoonright B_{\nu}^0$ is a bijection with $B_{\nu}^1$;

\item $\phi \upharpoonright B_{\nu}^2$ is the bijection with $B_{\nu}^3$ such that for $x\in B_{\nu}^2$ we have $\phi\circ(\psi_{\nu}\upharpoonright X_{0, \nu})^{-1}\circ \phi(x) = \psi_{\nu}(x)$;

\end{itemize}

\noindent Now define

\begin{center}

$D_{\eta, \nu} = \bigcup_{\xi \in Z_{\eta, \nu}} D_{\xi, \nu} \cup B_{\nu}^0 \cup B_{\nu}^1 \cup B_{\nu}^2 \cup B_{\nu}^3$

\end{center}

\noindent We point out that the sets $B_{\nu}^0, B_{\nu}^1, B_{\nu}^2, B_{\nu}^3$ are pairwise disjoint by construction.  Because $D_{\eta, \nu} = \underline{D_{\eta, \nu}} \cup B_{\nu}^0 \cup B_{\nu}^1 \cup B_{\nu}^2 \cup B_{\nu}^3$ we have $|D_{\eta, \nu}| < f(\nu)$.

Note that $A_{\eta, \nu} \cup \psi_{\nu}(A_{\eta, \nu})$ is disjoint from $\bigcup_{\xi \in Z_{\eta, \nu}} D_{\xi, \nu}$.  To see this, given $\xi \in Z_{\eta, \nu}$ we have by induction hypothesis 1(d) that $D_{\xi, \nu} \in \B_{\xi, \nu}$, and we have already seen that if $\xi \in \xi^* \in \B_{\eta, \nu}$ we have $\B_{\xi, \nu} \subseteq \B_{\xi, \nu}$, so in particular $D_{\xi, \nu}$ is a union of elements in $\mathcal{J}_{\eta, \nu}$.  As $|A_{\eta, \nu}| = f(\nu) > |D_{\xi, \nu}|$ and $A_{\eta, \nu}$ is in the partition $\mathcal{A}_{\eta, \nu}$ we see that $A_{\eta, \nu}\cap D_{\xi, \nu} = \emptyset$.  But also $\psi_{\nu}(A_{\eta, \nu}) \in \mathcal{A}_{\eta, \nu}$ and as $|\psi_{\nu}(A_{\eta, \nu})| = f(\nu)$ we argue as before that $\psi_{\nu}(A_{\eta, \nu}) \cap D_{\xi, \nu} = \emptyset$.

Thus it is also the case that each of the sets $B_{\nu}^0, B_{\nu}^1, B_{\nu}^2, B_{\nu}^3$ is disjoint from $\bigcup_{\xi \in Z_{\eta, \nu}} D_{\xi, \nu}$.  Let 

\begin{center}

$\phi_{\eta, \nu} = \bigcup_{\xi \in Z_{\eta, \nu}}\phi_{\xi, \nu} \cup \phi$

\end{center}

By construction $\phi_{\eta, \nu}\upharpoonright (D_{\eta, \nu} \cap I_{0, \nu})$ is a bijection with $D_{\eta, \nu} \cap I_{1, \nu}$ and $\phi_{\eta, \nu} \upharpoonright (D_{\eta, \nu} \cap I_{1, \nu})$ is identity, and for $x \in D_{\eta, \nu} \cap I_{0, \nu}$ we have $\phi_{\eta, \nu} \circ (\psi_{\nu} \upharpoonright I_{0, \nu})^{-1}\circ \phi_{\eta, \nu}(x) = \psi_{\nu}(x)$.  We have also seen that $|D_{\eta, \nu}| < f(\nu)$.  Thus 1(a) and 1(b) hold.

Let $\B_{\eta, \nu}$ be the Boolean subalgebra of $\Po(I_{\nu})$ generated by $\mathcal{A}_{\eta, \nu} \cup \{B_{\nu}^0, B_{\nu}^1, B_{\nu}^2, B_{\nu}^3\}$.  It is easy to see by our construction that $\B_{\eta, \nu}$ is atomic with 

$$
\begin{array}{ll}
\Atoms(\B_{\eta, \nu}) & = (\mathcal{A}_{\eta, \nu} \setminus \{A_{\eta, \nu}, \psi_{\nu}(A_{\eta, \nu})\}) \vspace*{2mm}\\
 & \text{   } \cup\{B_{\nu}^0, B_{\nu}^2, A_{\eta, \nu} \setminus (B_{\nu}^0 \cup B_{\nu}^2), B_{\nu}^1, B_{\nu}^3, \psi_{\nu}(A_{\eta, \nu}) \setminus (B_{\nu}^1 \cup B_{\nu}^3)\}.
\end{array}
$$

\noindent For $A \in \Atoms(\B_{\eta, \nu})$ we know if $A \subseteq I_{0, \nu}$ then $\psi_{\nu}(A) \in \Atoms(\B_{\eta, \nu})$ and if $A \subseteq I_{1, \nu}$ then $(\psi_{\nu} \upharpoonright I_{0, \nu})^{-1}(A) \in \Atoms(\B_{\eta, \nu})$ and $\psi_{\nu}(A) = A$.  Also, $$\bigcup_{A \subseteq I_{0, \nu}, A \in \Atoms(\B_{\eta, \nu})} A = I_{0, \nu}$$ so $\B_{\psi, \nu} \subseteq \B_{\eta, \nu}$ and by what we have said regarding atoms we have $X \in \B_{\eta, \nu}$ implies $\psi_{\nu}(X), \psi_{\nu}^{-1}(X) \in \B_{\eta, \nu}$ so inductive hypothesis 1(c) holds.  It is also clear that if $A \in \Atoms(\B_{\eta, \nu})$ with $A \subseteq D_{\eta, \nu} \cap I_{0, \nu}$ then $\phi_{\eta, \nu}(A) \in \Atoms(\B_{\eta, \nu})$, and for $A \in \Atoms(\B_{\eta, \nu})$ with $A \subseteq D_{\eta, \nu} \cap I_{1, \nu}$ we have $(\phi_{\eta, \nu} \upharpoonright (D_{\eta, \nu} \cap I_{0, \nu}))^{-1}(A) \in \Atoms(\B_{\eta, \nu})$.

We note that for $\xi \in Z_{\eta, \nu}$ we have $\B_{\xi, \nu} \subseteq \B_{\eta, \nu}$, i.e. 2(a) holds as well, since by construction $\mathcal{J}_{\eta, \nu} \subseteq \B_{\eta, \nu}$.

For 1(d), we recall that if $\xi, \xi^* \in Z_{\eta, \nu}$ with $\xi \in \xi^*$ we have $\xi \in Z_{\eta, \nu} \cap \xi^* = Z_{\xi^*, \nu}$, and by induction hypotheses 1(c) and 2(a) we have $\B_{\xi, \nu} \subseteq \B_{\xi, \nu}$ and both are atomic.  Therefore when $\xi^* \in Z_{\eta, \nu}$ we have $D_{\xi^*, \nu} \in \B_{\xi^*, \nu}$ (by hypothesis 1(d)) and also $\bigcup_{\xi \in Z_{\xi^*, \nu}} D_{\xi, \nu} \in \B_{\xi^*, \nu}$, so $D_{\xi^*, \nu} \setminus (\bigcup_{\xi \in Z_{\xi^*, \nu}} D_{\xi, \nu}) \in B_{\xi^*, \nu}$.  Therefore

$$D_{\eta, \nu} = B_{\nu}^0 \cup B_{\nu}^1\cup B_{\nu}^2 \cup B_{\nu}^3 \cup (\bigcup_{\xi^* \in Z_{\eta, \nu}} (D_{\xi^*, \nu} \setminus \bigcup_{\xi \in Z_{\xi^*, \nu}} D_{\xi, \nu})) \in \B_{\eta, \nu}$$

For 1(e) we point out that $|\Atoms(\B_{\eta, \nu})| = |\mathcal{A}_{\eta, \nu}| + 4 \leq 2^{g(\nu)}$, and it is quite clear that $\phi_{\xi, \nu} \subseteq \phi_{\eta, \nu}$ for $\xi \in Z_{\eta, \nu}$, so that 2(b) holds as well.

We must check 2(c).  Let $\xi \in Z_{\eta, \nu}$ be given, along with $A \in \B_{\xi, \nu}$.  Then $A \cap A_{\eta, \nu} = A_{\eta, \nu}$ or $A \cap A_{\eta, \nu} = \emptyset$, and similarly $A \cap \psi_{\nu}(A_{\eta, \nu}) = \psi_{\nu}(A_{\eta, \nu})$.  Therefore $A \cap (B_{\nu}^0 \cup B_{\nu}^2) = (B_{\nu}^0 \cup B_{\nu}^2)$ or $A \cap (B_{\nu}^0 \cup B_{\nu}^2) = \emptyset$, and similarly $A \cap (B_{\nu}^1 \cup B_{\nu}^3) = B_{\nu}^1 \cup B_{\nu}^3$ or $A \cap (B_{\nu}^1 \cup B_{\nu}^3) = \emptyset$.  Therefore $\phi(A \cap (B_{\nu}^0 \cup B_{\nu}^1 \cup B_{\nu}^2\cup B_{\nu}^3)) = \psi_{\nu}(A \cap (B_{\nu}^0 \cup B_{\nu}^1 \cup B_{\nu}^2\cup B_{\nu}^3))$.

We also have for $\xi^*  \in Z_{\eta, \nu} \setminus \xi$ that $D_{\xi^*, \nu} \setminus D_{\xi, \nu} \in \B_{\xi^*, \nu}$ and applying induction hypothesis 2(c) we obtain $$\phi_{\xi^*, \nu}(A \cap (D_{\xi^*, \nu} \setminus D_{\xi, \nu})) = \psi_{\nu}(A \cap (D_{\xi^*, \nu} \setminus D_{\xi, \nu})).$$  Thus

$$
\begin{array}{ll}
\phi_{\eta, \nu}(A \cap D_{\eta, \nu}) & = \phi_{\xi, \nu}(A \cap D_{\xi, \nu}) \cup \bigcup_{\xi^* \in Z_{\eta, \nu} \setminus \xi} \phi_{\xi^*, \nu}(A \cap (D_{\xi^*, \nu} \setminus D_{\xi, \nu}))\vspace*{2mm}\\

& \text{   }\cup \phi(A \cap (B_{\nu}^0 \cup B_{\nu}^1 \cup B_{\nu}^2\cup B_{\nu}^3))\vspace*{2mm}\\
& = \phi_{\xi, \nu}(A \cap D_{\xi, \nu}) \cup \bigcup_{\xi^* \in Z_{\eta, \nu} \setminus \xi} \psi_{\nu}(A \cap (D_{\xi^*, \nu} \setminus D_{\xi, \nu}))   \vspace*{2mm}\\
& \text{   }\cup \psi_{\nu}(A \cap (B_{\nu}^0 \cup B_{\nu}^1 \cup B_{\nu}^2\cup B_{\nu}^3))\vspace*{2mm}\\

& = \phi_{\xi, \nu}(A \cap D_{\xi, \nu}) \cup \psi_{\nu}(A \cap (D_{\eta, \nu} \setminus D_{\xi, \nu}))

\end{array}
$$

\noindent and so

$$
\begin{array}{ll}
\Phi_{\eta, \nu}(A) & = \phi_{\eta, \nu}(A \cap D_{\eta, \nu}) \cup \psi_{\nu}(A \setminus D_{\eta, \nu})\vspace*{2mm}\\
& = \phi_{\xi, \nu}(A \cap D_{\xi, \nu}) \cup \psi_{\nu}(A \cap (D_{\eta, \nu} \setminus D_{\xi, \nu})) \cup \psi_{\nu}(A \setminus D_{\eta, \nu})\vspace*{2mm}\\
& = \phi_{\xi, \nu}(A \cap D_{\xi, \nu}) \cup \psi_{\nu}(A \setminus D_{\xi, \nu})\vspace*{2mm}\\
& = \Phi_{\xi, \nu}(A)
\end{array}
$$

\noindent and so $\Phi_{\xi, \nu} \subseteq \Phi_{\eta, \nu}$.  The fact that $\Phi_{\eta, \nu}$ is a $\jmath$-mapping is immediate from our observations on how $\phi_{\eta, \nu}$ and $\psi_{\nu}$ behave on $\Atoms(\B_{\eta, \nu})$ and so 2(c) is holds, and all induction hypotheses hold.

Now we check that properties (1) - (4) in the statement of the lemma hold.  Property (1) is immediate from the construction.  For property (2), we let $\xi \in \eta \in \kappa^+$ be given.  By condition (4) of Lemma \ref{theZ} we select $\beta \in \kappa$ so that $\xi \in Z_{\eta, \beta}$, and by condition (3) of Lemma \ref{theZ} we see that for $\nu \in \kappa \setminus \beta$ the membership $\xi \in Z_{\eta, \nu}$ holds.  By induction hypotheses (2) (a) and (2) (b) we get that $\B_{\xi, \nu} \subseteq \B_{\eta, \nu}$ and $\Phi_{\xi, \nu} \subseteq \Phi_{\eta, \nu}$ hold for $\nu \in \kappa \setminus \beta$.

For property (3) of our lemma we let $A \in \Po(\kappa)$ be given.  Let $A \cap I_{\nu} = E_{\theta_{\nu}, \nu}$ for each $\nu \in \kappa$.  By the selection of $G$ we pick $\xi \in \kappa^+$ and $\beta \in \kappa$ such that $\theta_{\nu} \in G_{\xi}(\nu)$ for all $\nu \in \kappa \setminus \beta$.  Then by construction we have $E_{\theta_{\nu}, \nu} \in \B_{\xi, \nu}$ for all $\nu \in \kappa \setminus \beta$.  The lemma is proved.

For property (4) of our lemma we let $\{F_{\alpha}\}_{\alpha \in \kappa}$ be a collection of sets such that $|F_{\alpha}| < f(p(\alpha))$ and $F_{\alpha} \subseteq I_{0, p(\alpha)}$.  Then for $\nu \in \kappa$ we have the function $L_{\nu}: I_{1, \nu} \rightarrow [I_{0, \nu}]^{< f(\nu)}$ given by $L_{\nu}(\alpha) = F_{\alpha}$.  Let $\{\theta_{\nu}\}_{\nu \in \kappa}$, where $\theta_{\nu} \in  2^{f(\nu)}$, be the sequence such that $L_{\theta_{\nu}, \nu} = L_{\nu}$.  Then by the selection of $G$ there is $\xi^* \in \kappa^+$ and $\beta^* \in \kappa$ such that $\theta_{\nu} \in G_{\xi^*}(\nu)$ for all $\nu \in \kappa \setminus \beta^*$.  Then by construction we see that for $\nu \in \kappa \setminus \beta^*$ we have nonempty pairwise disjoint atoms $B_{\nu}^0, B_{\nu}^1 \in \B_{\xi, \nu}$ with $B_{\nu}^0 \subseteq I_{0, \nu}$, $B_{\nu}^1 \subseteq I_{1, \nu}$, $\Phi_{\xi, \nu}(B_{\nu}^0) = B_{\nu}^1$, and $(\bigcup_{\alpha \in B_{\nu}^1} L_{\theta_{\nu}, \nu}(x)) \cap B_{\nu}^0 = \emptyset$.  Select $\alpha_{\nu} \in B_{\nu}^1$.  By statement (4) we select $\xi^{**}\in \kappa^+$ and $\beta^{**} \in \kappa \setminus \beta^{*}$ such that $F_{\alpha_{\nu}} \in \B_{\xi^{**}, \nu}$ for each $\nu \in \kappa \setminus \beta^{**}$.  Let $\xi = \max\{\xi^*, \xi^{**}\} + 1$ and by conditions (3) and (4) of Lemma \ref{theZ} select $\beta \in \kappa\setminus \beta^{**}$ such that $\xi^{*}, \xi^{**} \in Z_{\xi, \beta}$.  Then by induction hypothesis (2)(a) and (b) we have for $\nu \in \kappa \setminus \beta$ that $\B_{\xi^*, \nu}, \B_{\xi^{**}, \nu} \subseteq \B_{\xi, \nu}$ and $\Phi_{\xi^{*}, \nu} \subseteq \Phi_{\xi, \nu}$.  Then in particular, for $\nu \in \kappa \setminus \beta$ we have $\alpha_{\nu} \notin \Phi_{\xi, \nu}(F_{\alpha})$ since $F_{\alpha_{\nu}} \subseteq I_{0, \nu} \setminus B_{\nu}^0$, $\Phi_{\xi^*, \nu}$ is a $\jmath$-mapping, $B_{\nu}^0 \in \Atoms(\B_{\xi^*, \nu})$, and $\Phi_{\xi, \nu} \supseteq \Phi_{\xi^*, \nu}$.

\end{proof}

We remind the reader of some concepts introduced in the introduction and give some observations.

\begin{definition}  If $\kappa$ is a regular cardinal we say functions $\sigma_0, \sigma_1 \in {}^{\kappa}\chi$ are \emph{eventually equal} if $\{\alpha \in \kappa: \sigma_0(\alpha) \neq \sigma_1(\alpha)\}\equiv_{\kappa} \emptyset$.    Eventual equality determines an equivalence relation which we will write as $\sigma_0 \sim \sigma_1$ and let $^{\kappa}\chi/ \sim$ denote the set of equivalence classes.  If $|\chi| < \kappa$ then $\sigma_0 \sim \sigma_1$ if and only if for all $x \in \chi$ we have $\sigma_0^{-1}(\{x\}) \equiv_{\kappa} \sigma_1^{-1}(\{x\})$.  We let $p_{\sim}$ denote the function $\sigma \mapsto [\sigma]_{\sim}$.  When $\chi$ is the underlying set for a universal algebra $\mathfrak{U}$, then ${}^{\kappa}\chi$ is the underlying set for the algebra $\mathfrak{U}^{\kappa}$ and the function $p_{\sim}$ is a homomorphism from $\mathfrak{U}^{\kappa}$ to $\mathfrak{U}^{\kappa}/\sim$. 
\end{definition}

\begin{definition}\label{strongsupport}  Let $\mathfrak{U} = (\chi, \mathcal{S})$ and $\mathfrak{M} = (\Omega, \mathcal{S})$ universal algebras and $\Theta: \mathfrak{U}^Z \rightarrow \mathfrak{M}$ a homomorphism from the power $\mathfrak{U}^Z$.  A subset $Y \subseteq Z$ is a \emph{strong support} for $\Theta$ if $\sigma \upharpoonright Y = \tau \upharpoonright Y$ implies $\Theta \sigma = \Theta \tau$ and for any $y \in Y$ and $\sigma \in \mathfrak{U}^Z$ there exists some $x \in \chi$ such that $\tau \in \mathfrak{U}^Z$ given by 

\[
\tau(z) = \left\{
\begin{array}{ll}
\sigma(z)
                                            & \text{if } z \neq y , \\
x                                        & \text{if } z = y
\end{array}
\right.
\]

\noindent has $\Theta \sigma \neq \Theta \tau$.

\end{definition}

\begin{obs}  We point out that a strong support, if it exists, is unique provided $\chi \neq \emptyset$.  To see this, suppose that $Y_0$ and $Y_1$ are strong supports for $\Theta$.  If $y \in Y_0 \setminus Y_1$, then in particular $Z \neq \emptyset$ and as $\chi \neq \emptyset$ we can select $\sigma \in \mathfrak{U}^Z$.  As $Y_0$ is a strong support pick $x \in \chi$ such that 

\[
\tau(z) = \left\{
\begin{array}{ll}
\sigma(z)
                                            & \text{if } z \neq y , \\
x                                        & \text{if } z = y
\end{array}
\right.
\]

\noindent has $\Theta \sigma \neq \Theta \tau$, but as $Y_1$ is a strong support and $\tau \upharpoonright Y_0 = \sigma \upharpoonright Y_1$  we have $\Theta \sigma = \Theta \tau$, contradiction.

\end{obs}

\begin{obs}\label{whenempty}  A homomorphic function $\Theta: \mathfrak{U}^Z \rightarrow \mathfrak{M}$ with nonempty domain and strong support $Y$ is constant if and only $Y = \emptyset$.  If $y \in Y$ then we select $\sigma$ in the domain of $\Delta$ and select $x \in \chi$ such that letting $\tau(y) = x$ and $\tau \upharpoonright (Z \setminus \{y\}) = \sigma \upharpoonright (Z \setminus \{y\})$ gives $\Theta \tau \neq \Theta \sigma$, so $\Delta$ is not constant.  On the other hand if $Y = \emptyset$ then for $\sigma$ and $\tau$ in the domain of $\Theta$ we get $\sigma\upharpoonright \emptyset = \tau \upharpoonright \emptyset$, so $\Theta\sigma = \Theta\tau$.

\end{obs}

\begin{lemma}\label{getoffeventually}  Suppose that $\kappa$ is an infinite regular cardinal, $\mathfrak{U} = (\chi, \mathcal{S})$ with $1 \leq |\chi| < \kappa$, and

$$\xymatrix{\mathfrak{U}^{\kappa} \ar[r]^{\Delta} \ar[d]^{p_{\sim}}&
\mathfrak{U}^{\kappa} \ar[d]^{p_{\sim}} \\ \mathfrak{U}^{\kappa}/\sim
\ar[r]^{\mathcal{E}} & \mathfrak{U}^{\kappa}/\sim    }$$

\noindent is a commuting diagram.  If $J \subseteq \kappa$ is such that $|J| = \kappa$ and $Y_{\alpha} \subseteq \kappa$ is a strong support for $\pi_{\alpha} \circ \Delta$ for each $\alpha \in J$, then $\bigcap_{\alpha \in J} Y_{\alpha} = \emptyset$.
\end{lemma}

\begin{proof}  Assume the hyotheses and suppose for contradiction that $\zeta \in \bigcap_{\alpha \in J} Y_{\alpha}$ .  Select $x_0 \in \chi$ and let $\sigma \in \mathfrak{U}^{\kappa}$ be constantly $x_0$.  For each $\alpha \in J$ there exists some $x_{\alpha} \in \chi$ such that if $\tau \upharpoonright (\kappa \setminus \{\zeta\}) = \sigma \upharpoonright (\kappa \setminus \{\zeta\})$ and $\tau(\zeta) = x_{\alpha}$ then $(\pi_{\alpha} \circ \Delta)\sigma \neq (\pi_{\alpha} \circ \Delta)\tau$.  As $\kappa$ is regular there exists some $J' \subseteq J$ and $x \in \chi$ with $|J'| = |J| = \kappa$ and for each $\alpha \in J'$ we have $(\pi_{\alpha} \circ \Delta)\sigma \neq (\pi_{\alpha} \circ \Delta)\tau$, where $\tau \upharpoonright (\kappa \setminus \{\zeta\}) = \sigma \upharpoonright (\kappa \setminus \{\zeta\})$ and $\tau(\zeta) = x$.  Then as $|J'| = \kappa$ we get 

\begin{center}

$\mathcal{E}[\sigma]_{\sim} = [\Delta\sigma]_{\sim} \neq [\Delta\tau]_{\sim} = \mathcal{E}[\tau]_{\sim}$

\end{center}

\noindent but of course $[\tau]_{\sim} = [\sigma]_{\sim}$, contradiction.

\end{proof}

\end{section}

\begin{section}{First construction}\label{Firstconstruction}

We are now prepared to define the first homomorphic projection.

\begin{construction}\label{thefunction}  Suppose that $\kappa$ is inaccessible and $\dagger(\kappa, f, g)$, and that $\mathfrak{U} = (\chi, \mathcal{S})$ is a universal algebra, with $2 \leq |\chi| < \kappa$.  Take $\{I_{\nu}\}_{\nu \in \kappa}$, $\{\B_{\xi, \nu}\}_{\xi\in \kappa^+, \nu \in \kappa}$, $\{\Phi_{\xi, \nu}\}_{\xi \in \kappa^+, \nu \in \kappa}$, $\{I_{0, \nu}\}_{\nu\in \kappa}$, and $\{I_{1, \nu}\}_{\nu \in \kappa}$ as in Lemma \ref{sequences}.

Given $\sigma \in \text{}^{\kappa}\chi$, for each $x \in \chi$ we select $\xi_x \in \kappa^+$ and $\beta_x \in \kappa$ such that $\sigma^{-1}(\{x\}) \cap I_{\nu} \in \B_{\xi_x, \nu}$ for all $\nu \in \kappa \setminus \beta_x$.  Let $\eta \in \kappa^+$ be greater than all elements in $\{\xi_x\}_{x\in \chi}$.  For each $x \in \chi$ select $\beta_x' \in \kappa$ such that $\B_{\xi_x, \nu} \subseteq \B_{\eta, \nu}$ and $\Phi_{{\xi_x, \nu}} \subseteq \Phi_{\eta, \nu}$ for all $\nu \in \kappa \setminus \beta_x'$.  Select $\beta \in \kappa$ which is greater than all elements in $\{\beta_x\}_{x \in \chi} \cup \{\beta_x'\}_{x \in \chi}$.  Then we have for all $\nu \in \kappa \setminus \beta$ and $x \in \chi$ that $\sigma^{-1}(\{x\}) \in \B_{\xi_x, \nu} \subseteq \B_{\eta, \nu}$ and $\Phi_{\xi_x, \nu} \subseteq \Phi_{\eta, \nu}$.  Let $\tau \in \text{}^{\kappa}\chi$ be such that $$\tau^{-1}(\{x\}) \cap I_{\nu} = (\sigma^{-1}(\{x\}) \cap I_{0, \nu}) \cup \Phi_{\eta, \nu}(\sigma^{-1}(\{x\}) \cap I_{0, \nu})$$ for all $\nu \in \kappa \setminus \beta$ (define $\tau \upharpoonright (\bigcup_{\nu \in \beta} I_{\nu})$ arbitrarily).

\end{construction}

\begin{lemma}\label{constructionwelldefined}  The assignment $\sigma \mapsto \tau$ in Construction \ref{thefunction} gives a well-defined $\mathcal{E}_0: \mathfrak{U}^{\kappa}/\sim \rightarrow \mathfrak{U}^{\kappa}/\sim$ which is a homomorphic projection.

\end{lemma}

\begin{proof}  To see that the element $\tau$ is well-defined up to $\sim$, suppose that instead we select $\overline{\xi_x} \in \kappa^+$ and $\overline{\beta_x} \in \kappa$ such that $f^{-1}(\{x\}) \cap I_{\nu} \in \B_{\overline{\xi_x}, \nu}$ for all $\nu \in \kappa \setminus \overline{\beta_x}$, select $\overline{\eta} \in \kappa^+$ greater than all elements in $\{\overline{\xi_x}\}_{x\in \chi}$, $\overline{\beta_x'}$ such that $\B_{\overline{\xi_x}, \nu} \subseteq \B_{\overline{\eta}, \nu}$  and $\Phi_{\overline{\xi_x}, \nu} \subseteq \Phi_{\overline{\eta}, \nu}$ for all $\nu \in \kappa \setminus \overline{\beta_x'}$, $\overline{\beta} \in \kappa$ greater than all elements in $\{\overline{\beta_x}\}_{x \in \chi} \cup \{\overline{\beta_x'}\}_{x \in \chi}$.  Define $\overline{\tau} \in \text{}^{\kappa}\chi$ by $$\overline{\tau}^{-1}(\{x\}) \cap I_{\nu} = (\sigma^{-1}(\{x\}) \cap I_{0, \nu}) \cup \Phi_{\overline{\eta}, \nu}(\sigma^{-1}(\{x\}) \cap I_{0, \nu})$$ for $\nu \in \kappa \setminus \beta$ and define $\overline{\tau} \upharpoonright (\bigcup_{\nu \in \overline{\beta}} I_{\nu})$ arbitrarily.

Now select $\underline{\eta} \in \kappa^+$ which is greater than both $\eta$ and $\overline{\eta}$.  Select $\underline{\beta} \in \kappa$ large enough that $\B_{\eta, \nu} \subseteq \B_{\underline{\eta} \nu}$, $\Phi_{\eta, \nu} \subseteq \Phi_{\underline{\eta}, \nu}$, $\B_{\overline{\eta}, \nu} \subseteq \B_{\underline{\eta} \nu}$, and $\Phi_{\overline{\eta}, \nu} \subseteq \Phi_{\underline{\eta}, \nu}$ for all $\nu \in \kappa \setminus \underline{\beta}$.  Then for $\nu \in \kappa \setminus \underline{\beta}$ we have 

$$
\begin{array}{ll}
\overline{\tau}^{-1}(\{x\}) \cap I_{\nu} & = (\sigma^{-1}(\{x\}) \cap I_{0, \nu}) \cup \Phi_{\overline{\eta}, \nu}(\sigma^{-1}(\{x\}) \cap I_{0, \nu})\vspace*{2mm}\\
& = (\sigma^{-1}(\{x\}) \cap I_{0, \nu}) \cup \Phi_{\eta, \nu}(\sigma^{-1}(\{x\}) \cap I_{0, \nu})\vspace*{2mm}\\
& = \tau^{-1}(\{x\})
\end{array}
$$

\noindent and so $\tau \sim \overline{\tau}$.

Notice also that if $\sigma \sim \sigma_1$ then we select $\overline{\beta} \in \kappa \setminus \beta$ such that $\nu \in \kappa \setminus \overline{\beta}$ implies $\sigma^{-1}(\{x\}) \cap I_{\nu} = \sigma_1^{-1}(\{x\}) \cap I_{\nu}$.  Then $\sigma_1^{-1}(\{x\}) \cap I_{\nu} \in \B_{\xi_x, \nu}$ for all $\nu \in \kappa \setminus \overline{\beta}$ and $x \in \chi$, and for all $\nu \in \kappa \setminus \alpha$ and $x \in \chi$ we have $\sigma_1^{-1}(\{x\}) \in \B_{\xi_x, \nu} \subseteq \B_{\eta, \nu}$ and $\Phi_{\xi_x, \nu} \subseteq \Phi_{\eta, \nu}$.  Then we define $\tau_1 \in \text{}^{\kappa}\chi$ to be such that $$\tau_1^{-1}(\{x\})\cap I_{\nu} = (\sigma_1^{-1}(\{x\}) \cap I_{0, \nu}) \cup \Phi_{\eta, \nu}(\sigma_1^{-1}(\{x\}) \cap I_{0, \nu})$$ for all $\nu \in \kappa \setminus \overline{\beta}$ and let $\tau_1 \upharpoonright (\bigcup_{\nu \in \overline{\beta}} I_{\nu})$ be arbitrary.  Then it is clear that $\tau_1 \upharpoonright (\bigcup_{\nu \in \kappa \setminus \overline{\beta}} I_{\nu}) = \tau \upharpoonright (\bigcup_{\nu \in \kappa \setminus \overline{\beta}} I_{\nu})$, so $\tau \sim \tau_1$.

We let $\mathcal{E}_0[\sigma]_{\sim} = [\tau]_{\sim}$.  We have seen that $\mathcal{E}_0$ is well-defined.  To see that $\mathcal{E}_0 \circ \mathcal{E}_0 = \mathcal{E}_0$, we point out that for $\sigma$ and $\tau$ as above we have $\tau \upharpoonright \bigcup_{\nu \in \kappa \setminus \beta} I_{0, \nu} = \sigma \upharpoonright \bigcup_{\nu \in \kappa \setminus \beta} I_{0, \nu}$, and by how $\mathcal{E}_0$ was defined it is immediate that $\mathcal{E}_0[\tau]_{\sim} = [\tau]_{\sim}$.

 We now check that $\mathcal{E}_0$ is a homomorphism.  Let $s(w_0, \ldots, w_{n-1}) \in \mathcal{S}$ be given, together with $[\sigma_0]_{\sim}, \ldots, [\sigma_{n-1}]_{\sim}$.  For each $x \in \chi$ select $\xi_x \in \kappa^+$ and $\beta_x \in \kappa$ such that $\sigma_i{-1}(\{x\}) \cap I_{\nu} \in \B_{\xi_x, \nu}$ for all $\nu \in \kappa \setminus \beta_x$ and $0 \leq i < n$.  Take $\eta \in \kappa^+$ greater than $\sup\{\xi_x\}_{x \in \chi}$.  For each $x \in \chi$ select $\beta_x' \in \kappa$ such that $\B_{\xi_x, \nu} \subseteq \B_{\eta, \nu}$ for all $\nu \in \kappa \setminus \beta_x'$.  Select $\beta \in \kappa$ greater than $\sup\{\beta_{x}\}_{x\in \chi} \cup \{\beta_x'\}_{x \in \chi}$.

For each $0 \leq i < n$ we let $\tau_i$ be given by $$\{\alpha \in I_{\nu}: \tau_i(\alpha) = x\} = (\{\alpha' \in I_{0, \nu}: \sigma_i(\alpha') = x\}) \cup \Phi_{\eta, \nu}(\{\alpha' \in I_{0, \nu}: \sigma_i(\alpha') = x\})$$ for $x \in \chi$ and $\nu \in \kappa \setminus \beta$, and $\tau_i\upharpoonright \bigcup_{\nu \in \beta} I_{\nu}$ is defined arbitrarily.  We have $\mathcal{E}_0[\sigma_i]_{\sim} = [\tau_i]_{\sim}$ for all $i$.  Let $\sigma$ be given by $\{\alpha \in I_{\nu}: \sigma(\alpha) = x\} =  \{\alpha \in I_{\nu}: s(\sigma_0(\alpha), \ldots, \sigma_{n-1}(\alpha)) = x\}$ for $x \in \chi$ and $\nu \in \kappa \setminus \beta$ and $\sigma\upharpoonright\bigcup_{\nu \in \beta} I_{\nu}$ is arbitrary.  We have $[\sigma]_{\sim} = s([\sigma_0]_{\sim}, \ldots, [\sigma_{n-1}]_{\sim})$.

Let $\tau$ be given by $\{\alpha \in I_{\nu}: \tau(\alpha) = x\} = \{\alpha \in I_{\nu}: s(\tau_0(\alpha), \ldots, \tau_{n-1}(\alpha)) = x\}$ for $x \in \chi$ and $\nu \in \kappa \setminus \beta$ and $\tau \upharpoonright\bigcup_{\nu \in \beta} I_{\nu}$ arbitrary, so $s([\tau_0]_{\sim}, \ldots, [\tau_{n-1}]_{\sim}) = [\tau]_{\sim}$.  Let $\overline{\tau}$ be given by $$\{\alpha \in I_{\nu}: \overline{\tau}(\alpha) = x\} = (\{\alpha' \in I_{0, \nu}: \sigma(\alpha') = x\} \cup \Phi_{\eta, \nu}(\{\alpha' \in I_{0, \nu}: \sigma(\alpha') = x\})$$ for $x \in \chi$ and $\nu \in \kappa \setminus \beta$ and $\overline{\tau}\upharpoonright\bigcup_{\nu \in \beta} I_{\nu}$ arbitrary, so that $\mathcal{E}_0[\sigma]_{\sim} = [\overline{\tau}]_{\sim}$.

For $x \in \chi$ we let $N_x = \{(x_0, \ldots, x_{n-1}) \in \chi^n: s(x_0, \ldots, x_{n-1}) = x\}$ and notice that for a fixed $x \in \chi$ and $\nu \in \kappa \setminus \beta$ we have

$$
\begin{array}{ll}
\{\alpha \in I_{\nu}: \tau(\alpha) = x\} & = \{\alpha \in I_{\nu}: s(\tau_0(\alpha), \ldots, \tau_{n-1}(\alpha)) = x\}\vspace*{2mm}\\
& = \{\alpha \in I_{\nu}: (\tau_0(\alpha), \ldots, \tau_{n-1}(\alpha)) \in N_x\}\vspace*{2mm}\\
& = \bigcup_{(x_0, \ldots, x_{n-1}) \in N_x} \bigcap_{i = 0}^{n-1}((\{\alpha' \in I_{0, \nu}: \sigma_i(\alpha') = x_i\})\vspace*{2mm}\\
& \text{      } \cup \Phi_{\eta, \nu}(\{\alpha' \in I_{0, \nu}: \sigma_i(\alpha') = x_i\}))\vspace*{2mm}\\
& = \bigcup_{(x_0, \ldots, x_{n-1}) \in N_x} (\{\alpha' \in I_{0, \nu}: \bigwedge_{i = 0}^{n-1}\sigma_i(\alpha') = x_i\}) \vspace*{2mm}\\
& \text{     } \cup \Phi_{\eta, \nu}(\{\alpha' \in I_{0, \nu}: \bigwedge_{i = 0}^{n-1}\sigma_i(\alpha') = x_i\})\vspace*{2mm}\\
& = (\{\alpha' \in I_{0, \nu}: \sigma(\alpha') = x\}) \cup \Phi_{\eta, \nu}(\{\alpha' \in I_{0, \nu}: \sigma(\alpha') = x\})\vspace*{2mm}\\
& = \{\alpha \in I_{\nu}: \overline{\tau}(\alpha) = x\}

\end{array}
$$

\noindent and so $[\tau]_{\sim} = [\overline{\tau}]_{\sim}$ and $\mathcal{E}_0$ is a homomorphism.  
\end{proof}

\begin{lemma}\label{nonemptystrong}  Suppose that $\Delta: \mathfrak{U}^{\kappa} \rightarrow \mathfrak{U}^{\kappa}$ such that $p_{\sim} \circ \Delta = \mathcal{E}_0 \circ p_{\sim}$.  If $Y_{\alpha}$ is a strong support for $\pi_{\alpha} \circ \Delta$ for each $\alpha \in J \subseteq \kappa$ then $|\{\alpha \in J: Y_{\alpha} = \emptyset\}| < \kappa$.
\end{lemma}

\begin{proof}  Suppose on the contrary that $|\{\alpha \in J: Y_{\alpha} = \emptyset\}| = \kappa$ and without loss of generality we replace $J$ with $\{\alpha \in J: Y_{\alpha} = \emptyset\}$.  Without loss of generality we further replace $J$ with a subset such that $p\upharpoonright J$ is injective while maintaining $|J|=\kappa$.  By Observation \ref{whenempty} for each $\alpha \in J$ the function $\pi_{\alpha} \circ \Delta$ is constant, say $\pi_{\alpha} \circ \Delta$ is constantly $x_{\alpha}$.  Using the fact that $|\chi| \geq 2$ we select $x_{\alpha}' \in \chi \setminus \{x_{\alpha}\}$.  Take $\sigma \in \mathfrak{U}^{\kappa}$ to have $\sigma \upharpoonright I_{p(\alpha)}$ be constantly $x_{\alpha}'$ and $\sigma$ is defined arbitrarily elsewhere.  Then letting $\tau \in \mathfrak{U}^{\kappa}$ be such that $\mathcal{E}_0[\sigma]_{\sim} = [\tau]_{\sim}$, it is easily seen by definition of $\mathcal{E}_0$ that 

\begin{center}

$\{\alpha \in J: \tau \upharpoonright I_{p(\alpha)} \text{ is constantly }x_{\alpha}'\} \equiv_{\kappa} J$

\end{center}

\noindent but on the other hand $(\pi_{\alpha} \circ \Delta)\sigma = x_{\alpha}$ for every $\alpha \in J$, so that $\mathcal{E}_0[\sigma]_{\sim} = [\tau]_{\sim} \neq [\Delta\sigma]_{\sim}$, contradicting $p_{\sim} \circ \Delta = \mathcal{E}_0 \circ p_{\sim}$.

\end{proof}

\begin{lemma}\label{strongsupportbehavior}  Suppose that we have a homomorphism $\Delta: \mathfrak{U}^{\kappa} \rightarrow \mathfrak{U}^{\kappa}$ such that $p_{\sim} \circ \Delta = \mathcal{E}_0 \circ p_{\sim}$.  Suppose also that $J \subseteq \kappa$ with $|J| = \kappa$ and for each $\nu \in J$ and each $\alpha \in I_{\nu}$ the homomorphism $\pi_{\alpha} \circ \Delta$ has strong support $Y_{\alpha}$ with $|Y_{\alpha}| < \kappa$.  Then there exists $J' \subseteq J$ with $|J'| = \kappa$ and for all $\nu \in J'$ and $\alpha \in I_{\nu}$ we have $Y_{\alpha} \subseteq I_{0, \nu}$.
\end{lemma}

\begin{proof}  Assume the hypotheses.  Since $\kappa$ is regular we know for each $\nu \in J$ that $|I_{\nu} \cup \bigcup_{\alpha \in I_{\nu}}Y_{\alpha}| < \kappa$.  Then by Lemma \ref{getoffeventually} it is straightforward to select by induction a $J'' \subseteq J$ such that the collection $\{I_{\nu} \cup \bigcup_{\alpha \in I_{\nu}}Y_{\alpha}\}_{\nu \in J''}$ is pairwise disjoint and $|J''| = \kappa$.  Let $J' = \{\nu \in J'': (\forall \alpha \in I_{\nu}) Y_{\alpha} \subseteq I_{0, \nu}\}$.

Suppose for contradiction that $|J'' \setminus J'| = \kappa$.  For each $\nu \in J'' \setminus J'$ we select $\alpha_{\nu} \in I_{\nu}$ such that $Y_{\alpha_{\nu}} \not\subseteq I_{0, \nu}$, so let $t_{\alpha_{\nu}} \in Y_{\alpha} \setminus I_{0, \nu}$.  Let $x_0 \in \chi$ and define $\sigma \in \mathfrak{U}^{\kappa}$ to be the function which is constantly $x_0$.  For each $t_{\alpha_{\nu}}$ select $x_{t_{\alpha_{\nu}}} \in \chi$ such that if $\tau$ agrees with $\sigma$, except $\tau(t_{\alpha_{\nu}}) = x_{t_{\alpha_{\nu}}}$, then $(\pi_{\alpha_{\nu}} \circ \Delta)\tau \neq (\pi_{\alpha_{\nu}} \circ \Delta)\sigma$.  Then more generally if $\tau \in \mathfrak{U}^{\kappa}$ is such that $\tau \upharpoonright Y_{\alpha} \setminus \{t_{\alpha_{\nu}}\}$ is constantly $x_0$ and $\tau(t_{\alpha_{\nu}}) = x_{t_{\alpha_{\nu}}}$ then $(\pi_{\alpha_{\nu}} \circ \Delta)\tau \neq (\pi_{\alpha_{\nu}} \circ \Delta)\sigma$.  Define $\tau_0 \in \mathfrak{U}^{\kappa}$ by 

\[
\tau_0(\gamma) = \left\{
\begin{array}{ll}
x_{t_{\alpha_{\nu}}}
                                            & \text{if } \gamma =  t_{\alpha_{\nu}} , \\
x_0                                        & \text{otherwise. }
\end{array}
\right.
\]

\noindent If we take $\tau_1 \in \mathfrak{U}^{\kappa}$ to be such that $[\tau_1]_{\sim} = \mathcal{E}_0[\tau_0]_{\sim}$ then by the construction of $\mathcal{E}_0$ it is the case that $$\{\nu \in J'' \setminus J': (\forall \alpha \in I_{\nu}) \tau_1(\alpha) = x_0\} \equiv_{\kappa} J'' \setminus J'$$ but on the other hand we know that $(\pi_{\alpha_{\nu}} \circ \Delta) \tau_0 \neq x_0$ for all $\nu \in J'' \setminus J'$, so that $[\Delta\tau_0]_{\sim} \neq [\tau_1]_{\sim} = \mathcal{E}_0[\tau_0]_{\sim}$ contradicting the assumption that $p_{\sim} \circ \Delta = \mathcal{E}_0 \circ p_{\sim}$.  Therefore $|J'' \setminus J'| < \kappa$, so $|J'| = \kappa$ and the lemma is proved.
\end{proof}

\begin{lemma}\label{Putittogether}  Suppose $\Delta: \mathfrak{U}^{\kappa} \rightarrow \mathfrak{U}^{\kappa}$ is a homomorphism such that $p_{\sim} \circ \Delta = \mathcal{E}_0 \circ p_{\sim}$.  For each $J \subseteq \kappa$ with $|J| = \kappa$ there exists $J' \subseteq J$ with $|J'| = \kappa$ and such that for every $\nu \in J'$ there exists $\alpha_{\nu} \in I_{\nu}$ such that the homomorphism $\pi_{\alpha_{\nu}} \circ \Delta$ does not have a strong support $Y_{\alpha_{\nu}}$ with $|Y_{\alpha_{\nu}}| < f(\nu)$.
\end{lemma}

\begin{proof}  Suppose that the conclusion fails.  Then there is a $J \subseteq \kappa$ with $|J| = \kappa$ and for every $\nu \in J$ and $\alpha \in I_{\nu}$ the homomorphism $\pi_{\alpha} \circ \Delta$ has a strong support $Y_{\alpha}$ with $|Y_{\alpha}| < f(\nu)$.  By removing fewer than $\kappa$ elements in $J$ we without loss of generality replace $J$ with a subset such that for every $\nu \in J$ and $\alpha \in I_{\nu}$ we have $Y_{\alpha} \neq \emptyset$, by Lemma \ref{nonemptystrong}.  As $f(\nu) < \kappa$ for all $\nu \in \kappa$ we have by Lemma \ref{strongsupportbehavior} a $J' \subseteq J$ such that $|J'| = \kappa$ and for all $\nu \in J'$ and $\alpha \in I_{\nu}$ the inclusion $Y_{\alpha} \subseteq I_{0, \nu}$ holds.

Now we define a collection $\{F_{\alpha}\}_{\alpha \in \kappa}$ by letting $F_{\alpha} = Y_{\alpha}$ for $\nu \in J'$ and $\alpha \in I_{\nu}$ and for $\nu \in \kappa \setminus J'$ and $\alpha \in I_{\nu}$ we let $F_{\alpha} = \emptyset$.  For this collection $\{F_{\alpha}\}_{\alpha \in \kappa}$ we have $|F_{\alpha}| < f(p(\alpha))$ and $F_{\alpha} \subseteq I_{0, p(\alpha)}$.  By Lemma \ref{sequences} there are $\xi \in \kappa^+$ and $\beta \in \kappa$ and sequence $\{\alpha_{\nu}\}_{\nu \in \kappa \setminus \beta}$ for which $F_{\alpha_{\nu}} \in \B_{\xi, \nu}$ and $\alpha_{\nu} \in I_{1, \nu} \setminus \Phi_{\xi, \nu}(F_{\alpha_{\nu}})$.  For each $\nu \in J'$ we select $t_{\alpha_{\nu}} \in F_{\alpha_{\nu}}$.  Let $x_0 \in \chi$ and $\sigma \in \mathfrak{U}^{\kappa}$ be the constant function with output $x_0$.  As $F_{\alpha_{\nu}}$ is a strong support for $\pi_{\alpha_{\nu}} \circ \Delta$ we pick $x_{t_{\alpha_{\nu}}} \in \chi$ such that for any $\tau \in \mathfrak{U}^{\kappa}$ such that $\tau \upharpoonright (F_{\alpha_{\nu}} \setminus \{t_{\alpha_{\nu}}\}) = \sigma \upharpoonright (F_{\alpha_{\nu}} \setminus \{t_{\alpha_{\nu}}\})$ and $\tau(t_{\alpha_{\nu}}) = x_{t_{\alpha_{\nu}}}$ we have $(\pi_{\alpha_{\nu}} \circ \Delta)\tau \neq (\pi_{\alpha_{\nu}} \circ \Delta)\sigma$.  Define $\tau_0 \in \mathfrak{U}^{\kappa}$ by

\[
\tau_0(\gamma) = \left\{
\begin{array}{ll}
x_{t_{\alpha_{\nu}}}
                                            & \text{if } \gamma =  t_{\alpha_{\nu}} , \\
x_0                                        & \text{otherwise. }
\end{array}
\right.
\]

Now by Lemma \ref{sequences} parts (2) and (3) select a $\xi^* \in \kappa^+ \setminus \xi$ and $\beta^* \in \kappa \setminus \beta$ such that for $\nu \in \kappa \setminus \beta^*$ we have

\begin{itemize}

\item $\{t_{\alpha_{\nu}}\} \in \B_{\xi^*, \nu}$; and

\item $\B_{\xi, \nu} \subseteq \B_{\xi^*, \nu}$ and $\Phi_{\xi, \nu} \subseteq \Phi_{\xi^*, \nu}$.

\end{itemize}

\noindent Let $\tau_1 \in \mathfrak{U}^{\kappa}$ be given by

\[
\tau_1(\gamma) = \left\{
\begin{array}{ll}
x_{t_{\alpha_{\nu}}}
                                            & \text{if } \gamma =  t_{\alpha_{\nu}}\text{ or } \Phi_{\xi^*, \nu}(\{t_{\alpha_{\nu}}\}) = \{\gamma\} \text{ with }\nu \in J' \setminus \beta^*, \\
x_0                                        & \text{otherwise. }
\end{array}
\right.
\]

\noindent By how $\mathcal{E}_0$ is defined we have $\mathcal{E}_0[\tau_0]_{\sim} = [\tau_1]_{\sim}$, but on the other hand we have

\begin{center}

$(\pi_{\alpha_{\nu}}\circ\Delta \tau)(\tau_0) = x_{t_{\alpha_{\nu}}} \neq x_0 = \pi_{\alpha_{\nu}}\tau_1$

\end{center}

\noindent for all $\nu \in J' \setminus \beta^*$, and as $|J' \setminus \beta^*| = \kappa$ we see that $[\Delta \tau_0]_{\sim} \neq [\tau_1]_{\sim} = \mathcal{E}_0[\tau_0]_{\sim}$, contradicting $p_{\sim} \circ \Delta = \mathcal{E}_0 \circ p_{\sim}$.
\end{proof}

\begin{lemma}\label{morespecifically}  Suppose that $\Delta: \mathfrak{U}^{\kappa} \rightarrow \mathfrak{U}^{\kappa}$ is a homomorphism such that $p_{\sim} \circ \Delta = \mathcal{E}_0 \circ p_{\sim}$ and that $\lambda < \kappa$.  The set

\begin{center}

$\{\alpha \in \kappa: \pi_{\alpha} \circ \Delta\text{ has no strong support of cardinality} \leq \lambda\}$

\end{center}

\noindent has cardinality $\kappa$.

\end{lemma}

\begin{proof}  Assume the hypotheses.  Recall that $\dagger(\kappa, f, g)$ implies that $\lim_{\nu \rightarrow \kappa}f(\nu) = \kappa$.  Take $J = \{\nu \in \kappa: \lambda < f(\nu)\}$, so $|J| = \kappa$ as $\kappa$ is regular.  By Lemma \ref{Putittogether} we obtain $J' \subseteq J$ with $|J'| = \kappa$ and for every $\nu \in J'$ we have $\alpha_{\nu} \in I_{\nu}$ such that $\pi_{\alpha_{\nu}}\circ \Delta$ does not have a strong support of cardinality $< f(\nu)$, in particular $\pi_{\alpha_{\nu}}\circ \Delta$ does not have support of cardinality $\leq \lambda$.
\end{proof}

\begin{lemma}\label{itisiso}  The image of $\mathcal{E}_0$ is isomorphic to $\mathfrak{U}^{\kappa}/\sim$.
\end{lemma}

\begin{proof}  For each $\nu \in \kappa$ we let $P_{\nu}: I_{0, \nu} \rightarrow I_{\nu}$ be a bijection and define bijection $P: \bigcup_{\nu \in \kappa} I_{0, \nu} \rightarrow \kappa$ by $\bigcup_{\nu \in \kappa}P_{\nu}$.  Let $H: \mathfrak{U}^{\kappa} \rightarrow \mathfrak{U}^{\kappa}$ be given by $(H(\sigma))(\alpha) = \sigma(P^{-1}(\alpha))$.  It is clear that $H$ is a homomorphism.  Letting $\overline{H}: \mathfrak{U}^{\kappa}/\sim \rightarrow \mathfrak{U}^{\kappa}/\sim$ be given by $\overline{H}([\sigma]_{\sim}) = [H(\alpha)]_{\sim}$ it is clear that $\overline{H}$ is also a homomorphism.

We claim that $\overline{H} \upharpoonright \im(\mathcal{E}_0)$ is an isomorphism from $\im(\mathcal{E}_0)$ to $\mathfrak{U}^{\kappa}$.  Letting $\tau \in \mathfrak{U}^{\kappa}$ be given, we take $\sigma \in \mathfrak{U}^{\kappa}$ to be such that $\sigma \upharpoonright I_{0, \nu}$ is given by $\sigma(\alpha) = \tau(P_{\nu}(\alpha))$ for each $\nu \in \kappa$ and let $\sigma \upharpoonright (\bigcup_{\nu \in \kappa}I_{1, \nu})$ be defined arbitrarily.  Then it is easy to see that $\overline{H}(\mathcal{E}_0[\sigma]_{\sim}) = [\tau]_{\sim}$.  Thus $\overline{H} \upharpoonright \im(\mathcal{E}_0)$ is onto $\mathfrak{U}^{\kappa}/\sim$.

To see that $\overline{H} \upharpoonright \im(\mathcal{E}_0)$ is injective, suppose that $[\sigma_0]_{\sim}, [\sigma_1]_{\sigma} \in \im(\mathcal{E}_0)$ is such that $\overline{H}([\sigma_0]_{\sim}) = \overline{H}([\sigma_1]_{\sim})$.  Then $H(\sigma_0) \sim H(\sigma_1)$, so in particular there exists some $\overline{\beta} \in \kappa$ such that $\nu \in \kappa \setminus \overline{\beta}$ implies $H(\sigma_0) \upharpoonright I_{\nu} = H(\sigma_1) \upharpoonright I_{\nu}$.  Then $\sigma_0 \upharpoonright I_{0, \nu} = \sigma_1 \upharpoonright I_{0, \nu}$ for all $\nu \in \kappa \setminus \overline{\beta}$.  Taking $\sigma \in \mathfrak{U}^{\kappa}$ to be such that $\sigma \upharpoonright I_{0, \nu} = \sigma_0 \upharpoonright I_{0, \nu}$ for all $\nu \in \kappa \setminus \overline{\beta}$, it is easy to see that $[\sigma_0]_{\sim} = \mathcal{E}_0[\sigma]_{\sim} = [\sigma_1]_{\sim}$.

\end{proof}

We anthologize the relevant facts.

\begin{theorem}\label{firstconstruction}  Suppose that $\kappa$ is inaccessible, $\dagger(\kappa, f, g)$, and that $\mathfrak{U} = (\chi, \mathcal{S})$ is a universal algebra with $2 \leq |\chi| < \kappa$.  Then there is a homomorphic projection $\mathcal{E}_0: \mathfrak{U}^{\kappa}/\sim \rightarrow \mathfrak{U}^{\kappa}/\sim$, with the image of $\mathcal{E}_0$ isomorphic to $\mathfrak{U}^{\kappa}/\sim$, such that any homomorphism $\Delta: \mathfrak{U}^{\kappa} \rightarrow \mathfrak{U}^{\kappa}$ making the following diagram commute

$$\xymatrix{\mathfrak{U}^{\kappa} \ar[r]^{\Delta} \ar[d]^{p_{\sim}}&
\mathfrak{U}^{\kappa} \ar[d]^{p_{\sim}} \\ \mathfrak{U}^{\kappa}/\sim
\ar[r]^{\mathcal{E}_0} & \mathfrak{U}^{\kappa}/\sim    }$$

\noindent has

\begin{center}

$\{\alpha \in \kappa: \pi_{\alpha} \circ \Delta\text{ has no strong support of cardinality} \leq \lambda\}$

\end{center}

\noindent of cardinality $\kappa$ for every $\lambda < \kappa$.

\end{theorem}

\end{section}

\begin{section}{A homomorphic projection using averages} \label{Secondconstruction}

We begin with a definition.

\begin{definition}  Let $\mathfrak{U} = (\chi, \mathcal{S})$ be a universal algebra.  We say that a homomorphism $\Av: \mathfrak{U} \times \mathfrak{U} \rightarrow \mathfrak{U}$ is an \emph{average} if for all $x_0, x_1 \in \chi$ we have $\Av(x_0, x_1) = \Av(x_1, x_0)$ and $\Av(x_0, x_0) = x_0$.

\end{definition}

\begin{construction}\label{usingaverage}  Suppose that $\kappa$ is inaccessible and $\dagger(\kappa, f, g)$ and that $\mathfrak{U} = (\chi, \mathcal{S})$ is a universal algebra with $2 \leq |\chi| < \kappa$ and that $\Av$ is an average for $\mathfrak{U}$.  Take $\{I_{\nu}\}_{\nu \in \kappa}$, $\{\B_{\xi, \nu}\}_{\xi\in \kappa^+, \nu \in \kappa}$, $\{\Phi_{\xi, \nu}\}_{\xi \in \kappa^+, \nu \in \kappa}$, and $\{I_{0, \nu}\}_{\nu\in \kappa}$ and $\{I_{1, \nu}\}_{\nu \in \kappa}$ as in Lemma \ref{sequences}.

Recall that for $\xi \in \kappa^+$ and $\nu \in \kappa$ the mapping $\Phi_{\xi, \nu}$ is a $\jmath$-mapping on $\B_{\xi, \nu}$, and particularly the restriction of $\Phi_{\xi, \nu}$ to $\{A \in \Atoms(\B_{\xi, \nu}): A \subseteq I_{0, \nu}\}$ gives a bijection with $\{A \in \Atoms(\B_{\xi, \nu}): A \subseteq I_{1, \nu}\}$.  Let $\gimel_{\xi, \nu}: \Atoms(\B_{\xi, \nu}) \rightarrow \Atoms(\B_{\xi, \nu})$ be the involution of atoms defined by

$$
\begin{array}{ll}
\gimel_{\xi, \nu} & = \Phi_{\xi, \nu} \upharpoonright \{A \in \Atoms(\B_{\xi, \nu}): A \subseteq I_{0, \nu}\}\vspace*{2mm}\\
& \text{      }\cup (\Phi_{\xi, \nu} \upharpoonright \{A \in \Atoms(\B_{\xi, \nu}): A \subseteq I_{0, \nu}\})^{-1}

\end{array}
$$

\noindent Given $\sigma \in \text{}^{\kappa}\chi$, for each $x \in \chi$ we select $\xi_x \in \kappa^+$ and $\beta_x \in \kappa$ such that $\sigma^{-1}(\{x\}) \cap I_{\nu} \in \B_{\xi_x, \nu}$ for all $\nu \in \kappa \setminus \beta_x$.  Let $\eta \in \kappa^+$ be greater than all elements in $\{\xi_x\}_{x\in \chi}$.  For each $x \in \chi$ select $\beta_x' \in \kappa$ such that $\B_{\xi_x, \nu} \subseteq \B_{\eta, \nu}$ and $\Phi_{{\xi_x, \nu}} \subseteq \Phi_{\eta, \nu}$ for all $\nu \in \kappa \setminus \beta_x'$.  Select $\beta \in \kappa$ which is greater than all elements in $\{\beta_x\}_{x \in \chi} \cup \{\beta_x'\}_{x \in \chi}$.  Then we have for all $\nu \in \kappa \setminus \beta$ and $x \in \chi$ that $\sigma^{-1}(\{x\}) \in \B_{\xi_x, \nu} \subseteq \B_{\eta, \nu}$ and $\Phi_{\xi_x, \nu} \subseteq \Phi_{\eta, \nu}$.  Let $\tau \in \text{}^{\kappa}\chi$ be defined by letting $\tau \upharpoonright A = \Av(x_0, x_1)$ for $A \in \Atoms(\B_{\eta, \nu})$, where $\sigma \upharpoonright A$ is constantly $x_0$ and $\sigma \upharpoonright \gimel_{\eta, \nu}(A)$ is $x_1$ (and define $\tau$ arbitrarily on $\bigcup_{\nu \in \eta} I_{\nu}$).

\end{construction}

\begin{lemma}\label{homindeed}  The assignment $\sigma \mapsto \tau$ given in Construction \ref{usingaverage} induces a well-defined homomorphic projection $\mathcal{E}_1: \mathfrak{U}^{\kappa}/\sim \rightarrow \mathfrak{U}^{\kappa}/\sim$.
\end{lemma}

\begin{proof}  We check well-definedness first.  Imagine that we select $\overline{\xi_x} \in \kappa^+$ and $\overline{\beta_x}$ such that $f^{-1}(\{x\}) \cap I_{\nu} \in \B_{\overline{\xi_x}, \nu}$ for all $\nu \in \overline{\beta_x}$, and select $\overline{\eta} \in \kappa^+$ with $\eta > \sup\{\overline{\xi_{x}}\}_{x \in \chi}$, set $\{\overline{\beta_x'}\}_{x \in \chi}$ with $\B_{\overline{\xi_x}, \nu} \subseteq \B_{\overline{\eta}, \nu}$ and $\Phi_{\overline{\xi_x}, \nu} \subseteq \Phi_{\overline{\eta}, \nu}$ for all $\nu \in \kappa \setminus \overline{\beta_x'}$ and $\overline{\beta} > \sup\{\overline{\beta_x}\}_{x \in \chi} \cup \{\overline{\beta_x'}\}_{x \in \chi}$.   Let $\overline{\tau}$ be given by $\overline{\tau}\upharpoonright A' = \Av(x_0, x_1)$ where $A' \in \Atoms(\B_{\overline{\eta}, })$ and $\sigma\upharpoonright A'$ is constantly $x_0$ and $\sigma \upharpoonright \gimel_{\overline{\eta}, \nu}(A')$ is $x_1$, and define $\overline{\tau} \upharpoonright \bigcup_{\nu \in \kappa \setminus \overline{\beta}}$ arbitrarily.

Select $\underline{\eta} \in \kappa^+$ with $\underline{\eta} > \eta, \overline{\eta}$ and $\underline{\beta} \in \kappa$ which is large enough that $\B_{\eta, \nu} \cup \B_{\overline{\eta}, \nu} \subseteq \B_{\underline{\eta}, \nu}$ and $\Phi_{\eta, \nu} \cup \Phi_{\overline{\eta}, \nu} \subseteq \Phi_{\underline{\eta}, \nu}$ for all $\nu \in \kappa \setminus \underline{\beta}$.  Now we fix $\nu \in \kappa \setminus \underline{\beta}$.  Take $A \in \Atoms(\B_{\eta, \nu})$ and $A' \in \Atoms(\B_{\overline{\eta}, \nu})$ such that $A \cap A' \neq \emptyset$.  We know $\sigma \upharpoonright A$ and $\sigma \upharpoonright A'$ are both constant, so $\sigma \upharpoonright A \cup A'$ is constant, say constantly $x_0$.  If $A \subseteq I_{0, \nu}$, and therefore $A' \subseteq I_{0, \nu}$, then $\Phi_{\underline{\eta}, \nu}(A) = \Phi_{\eta, \nu}(A) = \gimel_{\eta, \nu}(A)$ and $\Phi_{\underline{\eta}, \nu}(A') = \Phi_{\overline{\eta}, \nu}(A') = \gimel_{\overline{\eta}, \nu}(A')$.  Then $\sigma \upharpoonright \gimel_{\eta, \nu}(A)$ and $\sigma \upharpoonright \gimel_{\overline{\eta}, \nu}(A')$ are each constant functions, and $\gimel_{\eta, \nu}(A) \cap \gimel_{\overline{\eta}, \nu}(A') \neq \emptyset$, so $\sigma \upharpoonright \gimel_{\eta, \nu}(A) \cup \gimel_{\overline{\eta}, \nu}(A')$ is constant, say constantly $x_1$.  Therefore $\tau \upharpoonright (A \cap A')$ is constantly $\Av(x_0, x_1)$ and so is $\overline{\tau} \upharpoonright (A \cap A')$.  Then $\tau \upharpoonright (\bigcup_{\nu \in \kappa \setminus \overline{\beta}} I_{\nu}) = \overline{\tau} \upharpoonright (\bigcup_{\nu \in \kappa \setminus \overline{\beta}} I_{\nu})$ and $\tau \sim \overline{\tau}$.

When $\sigma \sim \sigma_1$, the element $\tau_1 \in \mathfrak{U}^{\kappa}$ with $\sigma_1 \mapsto \tau_1$ under Construction \ref{usingaverage} has $\tau \sim \tau_1$.  The check is straightforward and follows appropriate changes to the comparable claim in Lemma \ref{constructionwelldefined}.  Thus we have a well-defined function $\mathcal{E}_1$.

To see that $\mathcal{E}_1 \circ \mathcal{E}_1 = \mathcal{E}_1$ we take $\sigma$, $\eta$ and $\beta$ and $\tau$ as before.  It is clear that for every $x \in \chi$ and $\nu \in \kappa \setminus \beta$ we have $\tau^{-1}(\{x\}) \cap I_{\nu} \in \B_{\eta, \nu}$.  For each atom $A \in \Atoms(\B_{\eta, \nu})$ we have $\tau \upharpoonright A$ and $\tau \upharpoonright \gimel_{\eta, \nu}(A)$ are each constantly $x$, by how $\tau$ is constructed.  Then since $\Av(x, x) = x$, we see that again $\mathcal{E}_1[\tau]_{\sim} = [\tau]_{\sim}$, so $\mathcal{E}_1 \circ \mathcal{E}_1 = \mathcal{E}_1$.

Now we check that $\mathcal{E}_1$ is a homomorphism.  Let $s(w_0, \ldots, w_{n-1}) \in \mathcal{S}$ and also $[\sigma_0]_{\sim}, \ldots, [\sigma_{n-1}]_{\sim}$.  We take $\eta \in \kappa^+$ large enough and $\beta \in \kappa$ large enough that for every $x \in \chi$ and $\nu \in \kappa \setminus \beta$ we have $\sigma_{i}^{-1}(x) \cap I_{\nu} \in \B_{\eta, \nu}$.  For $A \in \Atoms(\B_{\eta, \nu})$ and $0 \leq i < n$ take $\tau_i \upharpoonright A$ to be constantly $\Av(x_0, x_1)$ where $\sigma_i \upharpoonright A$ and $\sigma_i \upharpoonright \gimel_{\eta, \nu}(A)$ are constantly $x_0$ and $x_1$, respectively and $\tau_i$ is defined arbitrarily elsewhere.  Then $\mathcal{E}_1([\sigma_i]_{\sim}) = [\tau_i]_{\sim}$.

Let $\sigma$ be defined by having $\sigma \upharpoonright A$ be constantly $x$ when $A \in \Atoms(\B_{\eta, \nu})$ where $\sigma_i \upharpoonright A$ is constantly $x_i$ and $s(x_0, \ldots, x_{n-1}) = x$, and define $\sigma$ arbitrarily elsewhere.  Thus $[\sigma]_{\sim} = s([\sigma_0]_{\sim}, \ldots, [\sigma_{n-1}]_{\sim})$.  Let $\tau$ be defined by having $\tau \upharpoonright A$ be constantly $x$ when $A \in \Atoms(\B_{\eta, \nu})$ where $\tau_i \upharpoonright A$ is constantly $x_i$ and $s(x_0, \ldots, x_{n-1}) = x$, and define $\tau$ arbitrarily elsewhere.  Thus $\mathcal{E}_1[\sigma]_{\sim} = [\tau]_{\sim}$.  Let $\overline{\tau}$ be defined by $\overline{\tau} \upharpoonright A$ to be constantly $\Av(x_0, x_1)$ where $\sigma \upharpoonright A$ and $\sigma \upharpoonright \gimel_{\eta, \nu}(A)$ are constantly $x_0$ and $x_1$, respectively and $\overline{\tau}$ is defined arbitrarily elsewhere.  Then $\mathcal{E}_1[\sigma]_{\sim} = [\overline{\tau}]_{\sim}$.

Fix $\nu \in \kappa \setminus \beta$ and $A \in \Atoms(\B_{\eta, \nu})$ we let $\sigma_i \upharpoonright A$ be constantly $x_{0, i}$ and $\sigma_i \upharpoonright \gimel_{\eta, \nu}(A)$ be constantly $x_{1, i}$.  Then $\tau_i \upharpoonright A$ is constantly $\Av(x_{0, }i, x_{1, i})$, and $\tau \upharpoonright A$ is constantly $s(\Av(x_{0, 0}, x_{1, 0}), \ldots, \Av(x_{0, n-1}, x_{1, n-1}))$, and $\overline{\tau} \upharpoonright A$ is constantly $\Av(s(x_{0, 0}, \ldots, x_{0, n-1}), s(x_{1, 0}, \ldots, x_{1, n-1}))$.  Then $\tau \upharpoonright A = \overline{\tau} \upharpoonright A$ since $\Av$ is a homomorphism, and so $\tau \sim \overline{\tau}$.

\end{proof}

\begin{lemma}\label{nonemptystrongsupp}  Suppose that $\Delta: \mathfrak{U}^{\kappa} \rightarrow \mathfrak{U}^{\kappa}$ is such that $p_{\sim} \circ \Delta = \mathcal{E}_1 \circ p_{\sim}$.  If $Y_{\alpha}$ is a strong support for $\pi_{\alpha} \circ \Delta$ for each $\alpha \in J \subseteq \kappa$ then $|\{\alpha \in J: Y_{\alpha} = \emptyset\}| < \kappa$.
\end{lemma}

\begin{proof}  Assume the hypotheses, and suppose for contradiction that $|\{\alpha \in J: Y_{\alpha} = \emptyset\}| = \kappa$.  Without loss of generality we can assume that $Y_{\alpha} \neq \emptyset$ for all $\alpha \in J$ and that $p \upharpoonright J$ is injective.  Define function $r: J \rightarrow \{0, 1\}$ by $\alpha \in I_{r(\alpha), p(\alpha)}$.  As $|\chi| \geq 2$ there exist $x_0, x_1 \in \chi$ such that $\Av(x_0, x_1) \neq x_0$.  Define $\sigma_0 \in \mathfrak{U}^{\kappa}$ to be constantly $x_0$ and define $\sigma_1 \in \mathfrak{U}^{\kappa}$ by

\[
\sigma_1(\gamma) = \left\{
\begin{array}{ll}
x_1
                                            & \text{if } \gamma \in  I_{1-r(\alpha), p(\alpha)} \text{ for some }\alpha \in J, \\
x_0                                        & \text{otherwise. }
\end{array}
\right.
\]

\noindent It is clear that $\mathcal{E}_1[\sigma_0]_{\sim} = [\sigma_0]_{\sim}$ and $\mathcal{E}_1[\sigma_1]_{\sim} = [\tau_1]_{\sim}$ where

\[
\tau_1(\gamma) = \left\{
\begin{array}{ll}
\Av(x_0, x_1)
                                            & \text{if } \gamma \in  I_{p(\alpha)} \text{ for some }\alpha \in J, \\
x_0                                        & \text{otherwise }
\end{array}
\right.
\]

\noindent and in particular $$|\{\alpha \in J: (\pi_{\alpha} \circ \Delta)\sigma_0 \neq (\pi_{\alpha} \circ \Delta)\sigma_1\}| = \kappa$$ whereas by Observation \ref{whenempty} $$\{\alpha \in J: \pi_{\alpha} \circ \Delta\text{ is constant}\} \equiv_{\kappa} J$$ which gives a contradiction.

\end{proof}

\begin{lemma}  \label{fallingintocorrectplace}  Suppose that we have a homomorphism $\Delta: \mathfrak{U}^{\kappa} \rightarrow \mathfrak{U}^{\kappa}$ such that $p_{\sim} \circ \Delta = \mathcal{E}_1 \circ p_{\sim}$.  Suppose also that $J \subseteq \kappa$ with $|J| = \kappa$ and for each $\nu \in J$ and each $\alpha \in I_{\nu}$ the homomorphism $\pi_{\alpha} \circ \Delta$ has strong support $Y_{\alpha}$ with $|Y_{\alpha}| < \kappa$.  Then there exists $J' \subseteq J$ with $|J'| = \kappa$ and for all $\nu \in J'$ and $\alpha \in I_{\nu}$ we have

\begin{enumerate}[(a)]

\item $Y_{\alpha} \subseteq I_{\nu}$; and

\item $Y_{\alpha} \cap I_{0, \nu} \neq \emptyset$.

\end{enumerate}

\end{lemma}

\begin{proof}  By the same proof as that of Lemma \ref{strongsupportbehavior} but with every mention of $Y_{0, \nu}$ being replaced with $Y_{\nu}$, we obtain a $J'' \subseteq J$ with $|J''| = \kappa$ and for all $\nu \in J''$ and $\alpha \in I_{\nu}$ we have $Y_{\alpha} \subseteq I_{\nu}$.

Now suppose for contradiction that $J_1 = \{\nu \in J'': (\exists \alpha \in I_{\nu}) Y_{\alpha} \cap I_{0, \nu} = \emptyset\}$ is of cardinality $\kappa$.  Select for each $\nu \in J_1$ a $\alpha_{\nu} \in I_{\nu}$ with $Y_{\alpha} \cap I_{0, \nu} = \emptyset$.  As $|\chi| \geq 2$ select $x_0, x_1 \in \chi$ such that $\Av(x_0, x_1) \neq x_0$.  Let $\sigma_0 \in \mathfrak{U}^{\kappa}$ be constantly $x_0$.  Let $\sigma_1 \in \mathfrak{U}^{\kappa}$ be given by

\[
\sigma_1(\gamma) = \left\{
\begin{array}{ll}
x_1
                                            & \text{if } \gamma \in  I_{0, \nu} \text{ for some }\nu \in J_0, \\
x_0                                        & \text{otherwise. }
\end{array}
\right.
\]

\noindent Clearly $\mathcal{E}_1[\sigma_0]_{\sim} = [\sigma_0]_{\sim}$ and $\mathcal{E}_1[\sigma_1]_{\sim} = [\tau_1]_{\sim}$ where

\[
\tau_1(\gamma) = \left\{
\begin{array}{ll}
\Av(x_0, x_1)
                                            & \text{if } \gamma \in  I_{\nu} \text{ for some }\nu \in J_1, \\
x_0                                        & \text{otherwise. }
\end{array}
\right.
\]

\noindent As $Y_{\alpha_{\nu}} \cap I_{0, \nu} = \emptyset$ for each $\nu \in J_1$ we see that $(\pi_{\alpha_{\nu}}\circ \Delta)\sigma_0 = (\pi_{\alpha_{\nu}} \circ \Delta)\sigma_1$, but on the other hand it is also the case that $\pi_{\alpha_{\nu}}\sigma_0  = x_0 \neq \Av(x_0, x_1) = \pi_{\alpha_{\nu}}\tau_1$ for all $\nu \in J_1$, contradicting $p_{\sim} \circ \Delta
 = \mathcal{E}_1 \circ p_{\sim}$.  Thus we let $J' = J'' \setminus J_1$ and the claim is true.

\end{proof}

\begin{lemma}  Suppose that $\Delta: \mathfrak{U}^{\kappa} \rightarrow \mathfrak{U}^{\kappa}$ is a homomorphism such that $p_{\sim} \circ \Delta = \mathcal{E}_1 \circ p_{\sim}$.  For each $J \subseteq \kappa$ with $|J| = \kappa$ there exists $J' \subseteq J$ with $|J'| = \kappa$ and such that for every $\nu \in J'$ there exists $\alpha_{\nu} \in J'$ such that the homomorphism $\pi_{\alpha_{\nu}} \circ \Delta$ does not have a strong support $Y_{\alpha_{\nu}}$ with $|Y_{\alpha_{\nu}}| < f(\nu)$.
\end{lemma}

\begin{proof}  Suppose for contradiction that there is some $J \subseteq \kappa$ with $|J| = \kappa$ such that for every $\nu \in J$ and $\alpha \in I_{\nu}$ the homomorphism $\pi_{\alpha} \circ \Delta$ has strong support $Y_{\alpha}$ with $|Y_{\alpha}| < f(\nu)$.  By Lemma \ref{nonemptystrongsupp} we can assume without loss of generality that $Y_{\alpha} \neq \emptyset$ for each $\alpha \in I_{\nu}$ with $\nu \in J$.  By Lemma \ref{fallingintocorrectplace} we have a $J' \subseteq J$ with $|J'| = \kappa$ and for all $\nu \in J'$ and $\alpha \in I_{\nu}$ we get $Y_{\alpha} \subseteq I_{\nu}$ and $Y_{\alpha} \cap I_{0, \nu} \neq \emptyset$.  Let collection $\{F_{\alpha}\}_{\alpha \in \kappa}$ be defined by $F_{\alpha} = Y_{\alpha} \cap I_{0, \nu}$ when $\alpha \in I_{\nu}$ and $\nu \in J'$ and let $F_{\alpha} = \emptyset$ when $\alpha \in I_{\nu}$ for $\nu \in \kappa \setminus J'$.  Now we have $|F_{\alpha}| < f(p(\alpha))$ and $F_{\alpha} \subseteq I_{0, p(\alpha)}$ for each $\alpha \in \kappa$.  Now by Lemma \ref{sequences} part (4) we have $\xi \in \kappa^+$ and $\beta \in \kappa$ and sequence $\{\alpha_{\nu}\}_{\nu \in \kappa \setminus \beta}$ such that for each $\nu \in \kappa \setminus \beta$ we have both $F_{\alpha_{\nu}} \in \B_{\xi, \nu}$ and $\alpha_{\nu} \in I_{1, \nu} \setminus \Phi_{\xi, \nu}(F_{\alpha_{\nu}})$.

For each $\nu \in J'$ pick $t_{\alpha_{\nu}} \in F_{\alpha_{\nu}} \cap I_{0, \nu}$.  Let $x_0 \in \chi$ and $\sigma \in \mathfrak{U}^{\kappa}$.  Since $F_{\alpha_{\nu}}$ is a strong support for $\pi_{\alpha_{\nu}} \circ \Delta$ it is possible to select $x_{t_{\alpha_{\nu}}} \in \chi$ such that when $\tau \in \mathfrak{U}^{\kappa}$ satisfies $\tau \upharpoonright (F_{\alpha_{\nu}}\setminus \{t_{\alpha_{\nu}}\}) = \sigma \upharpoonright (F_{\alpha_{\nu}}\setminus \{t_{\alpha_{\nu}}\})$ and $\tau(t_{\alpha_{\nu}}) = x_{t_{\alpha_{\nu}}}$ we have $(\pi_{\alpha_{\nu}} \circ \Delta)\tau \neq (\pi_{\alpha_{\nu}}\circ \Delta)\sigma$.  Let $\tau_0 \in \mathfrak{U}^{\kappa}$ be given by

\[
\tau_0(\gamma) = \left\{
\begin{array}{ll}
x_{t_{\alpha_{\nu}}}
                                            & \text{if } \gamma =t_{\alpha_{\nu}}, \\
x_0                                        & \text{otherwise. }
\end{array}
\right.
\]

\noindent By Lemma \ref{sequences} (2) and (3) we select $\xi^* \in \kappa^+ \setminus \xi$ and $\beta^* \in \kappa \setminus \beta$ such that for any $\nu \in \kappa \setminus \beta^*$ we have $\{t_{\alpha_{\nu}}\} \in \B_{\xi^*, \nu}$ and $\B_{\xi, \nu} \subseteq \B_{\xi^*, \nu}$ and $\Phi_{\xi, \nu} \subseteq \Phi_{\xi^*, \nu}$.  Take $\tau_1 \in \mathfrak{U}^{\kappa}$ to be

\[
\tau_0(\gamma) = \left\{
\begin{array}{ll}
\Av(x_{t_{\alpha_{\nu}}}, x_0)
                                            & \text{if } \gamma =t_{\alpha_{\nu}} \text{ or }\Phi_{\xi^*, \nu}(\{t_{\alpha_{\nu}}\}) = \{\gamma\}, \\
x_0                                        & \text{otherwise. }
\end{array}
\right.
\]

\noindent It is clear that $\mathcal{E}_1[\tau_0]_{\sim} = [\tau_1]_{\sim}$, but we also know $(\pi_{\alpha_{\nu}} \circ \Delta)\tau_0 \neq x_0 = \pi_{\alpha_{\nu}}\tau_1$ for all $\nu \in J' \setminus \beta^*$, and since $|J' \setminus \beta^*| = \kappa$ we get $[\Delta \tau_0]_{\sim} \neq [\tau_1]_{\sim} = \mathcal{E}_1[\tau_0]_{\sim}$, which contradicts $p_{\sim} \circ \Delta = \mathcal{E}_1 \circ p_{\sim}$.

\end{proof}

\begin{lemma}\label{abitmorespecific}  Suppose that $\Delta: \mathfrak{U}^{\kappa} \rightarrow \mathfrak{U}^{\kappa}$ is a homomorphism such that $p_{\sim} \circ \Delta = \mathcal{E}_1 \circ p_{\sim}$.  The set

\begin{center}

$\{\alpha \in \kappa: \pi_{\alpha} \circ \Delta \text{ has no strong support of cardinality of }\leq \lambda\}$

\end{center}

\noindent has cardinality $\kappa$ for every $\lambda < \kappa$.

\end{lemma}

\begin{proof}  The proof is that of Lemma \ref{morespecifically} with obvious modifications.

\end{proof}

\begin{lemma}\label{itisisomorphic}  The image of $\mathcal{E}_1$ is isomorphic to $\mathfrak{U}^{\kappa}/\sim$.
\end{lemma}

\begin{proof}  We begin as in Lemma \ref{itisiso}.  For $\nu \in \kappa$ let $P_{\nu}: I_{0, \nu} \rightarrow I_{\nu}$ be a bijection and let $P: \bigcup_{\nu \in \kappa} I_{0, \nu} \rightarrow \kappa$ be the bijection $\bigcup_{\nu \in \kappa}P_{\nu}$.  Then $H: \mathfrak{U}^{\kappa} \rightarrow \mathfrak{U}^{\kappa}$ is given by $(H(\sigma))(\alpha) = \sigma(P^{-1}(\alpha))$.  It is clear that $H$ is a homomorphism.  Let $\overline{H}: \mathfrak{U}^{\kappa}/\sim \rightarrow \mathfrak{U}^{\kappa}/\sim$ be given by $\overline{H}([\sigma]_{\sim}) = [H(\sim)]_{\sim}$ and $\overline{H}$ is also a homomorphism.

We'll show that $\overline{H} \upharpoonright \im(\mathcal{E}_1)$ is an isomorphism from $\im(\mathcal{E}_1)$ to $\mathfrak{U}^{\kappa}$.  Letting $\tau \in \mathfrak{U}^{\kappa}$ be given.  We'll define an element $\sigma \in \mathfrak{U}^{\kappa}$ in stages.  Let to be such that $\sigma \upharpoonright I_{0, \nu}$ is given by $\sigma(\alpha) = \tau(P_{\nu}(\alpha))$.  By Lemma \ref{sequences} part (3),  $|\chi| < \kappa$ , and the fact that $\kappa$ is regular we can select $\xi \in \kappa^+$ and $\beta \in \kappa$ large enough that $\nu \in \kappa \setminus \beta$ and $x \in \chi$ imply that $K_{x, \nu} := \{\alpha \in I_{0, \nu}: \sigma(\alpha) = x\} \in \B_{\xi, \nu}$.  Then $\Phi_{\xi, \nu}(K_{x, \nu}) \in \B_{\xi, \nu}$ for all $\nu \in \kappa \setminus \beta$.  For each $A \in \Atoms(\B_{\xi, \nu})$ with $A \subseteq I_1, \nu$ we let $\sigma \upharpoonright A$ be the constant $x_0$ where $\sigma\upharpoonright \gimel_{\xi, \nu}(A)$ is $x_0$.  Clearly $\mathcal{E}_1[\sigma]_{\sim} = [\sigma]_{\sim}$ and that $\overline{H}([\sigma]_{\sim}) = [\tau]_{\sim}$, so $\overline{H} \upharpoonright \im(\mathcal{E}_1)$ is onto $\mathfrak{U}^{\kappa}/\sim$.

To see that $\overline{H} \upharpoonright \im(\mathcal{E}_1)$ is injective, suppose that $[\sigma_0]_{\sim}, [\sigma_1]_{\sigma} \in \im(\mathcal{E}_1)$ is such that $\overline{H}([\sigma_0]_{\sim}) = \overline{H}([\sigma_1]_{\sim})$.  Then $H(\sigma_0) \sim H(\sigma_1)$.  Select $\xi \in \kappa^+$ and $\beta \in \kappa$ large enough so that for $x \in \chi$ and $\nu \in \kappa \setminus \beta$ we have $K_{x, \nu} := \sigma_0^{-1}(\{x\}) \cap I_{0, \nu} \in \B_{\xi, \nu}$.  Define $\sigma$ as in the previous paragraph and it is clear that $[\sigma]_{\sim} = \mathcal{E}_1[\sigma]_{\sim} = \mathcal{E}_1[\sigma_0]_{\sim} = \mathcal{E}_1[\sigma_1]_{\sim}$.  Then $\overline{H}$ is injective.
\end{proof}

\begin{definition}\label{changing}  We say an average $\Av: \mathfrak{U} \times \mathfrak{U} \rightarrow \mathfrak{U}$ is \emph{changing} provided there exist $x_0, x_1 \in \mathfrak{U}$ such that $x_0 \neq \Av(x_0, x_1) \neq x_1$.
\end{definition}

\begin{theorem}\label{combinginConstruction2}  Suppose that $\kappa$ is a inaccesible, $\dagger(\kappa, f, g)$, and that $\mathfrak{U} = (\chi, \mathcal{S})$ is a universal algebra with $2 \leq |\chi| < \kappa$ and that $\Av$ is an average for $\mathfrak{U}$.  Then there is a homomorphic projection $\mathcal{E}_1: \mathfrak{U}^{\kappa}/\sim \rightarrow \mathfrak{U}^{\kappa}/\sim$, with the image of $\mathcal{E}_1$ isomorphic to $\mathfrak{U}^{\kappa}/\sim$, such that any homomorphism $\Delta: \mathfrak{U}^{\kappa} \rightarrow \mathfrak{U}^{\kappa}$ making the following diagram commute

$$\xymatrix{\mathfrak{U}^{\kappa} \ar[r]^{\Delta} \ar[d]^{p_{\sim}}&
\mathfrak{U}^{\kappa} \ar[d]^{p_{\sim}} \\ \mathfrak{U}^{\kappa}/\sim
\ar[r]^{\mathcal{E}_1} & \mathfrak{U}^{\kappa}/\sim    }$$

\noindent has

\begin{center}

$\{\alpha \in \kappa: \pi_{\alpha} \circ \Delta\text{ has no strong support of cardinality} \leq \lambda\}$

\end{center}

\noindent of cardinality $\kappa$ for every $\lambda < \kappa$.  Moreover if $\Av$ is changing then there exists $\sigma \in \mathfrak{U}^{\kappa}$ such that when $\mathcal{E}_1[\sigma]_{\sim} = [\tau]_{\sim}$ we get $\{\alpha \in \kappa: \sigma(\alpha) = \tau(\alpha)\} \equiv_{\kappa} \emptyset$.
\end{theorem}

\begin{proof}  Everything has been checked already except the claim in the last sentence.  Selecting $x_0, x_1 \in \chi$ witnessing that $\Av$ is changing, we let $\sigma \in \mathfrak{U}$ be given by

\[
\sigma(\alpha) = \left\{
\begin{array}{ll}
x_0
                                            & \text{if } \alpha \in \bigcup_{\nu \in \kappa}I_{0, \nu}, \\
x_1                                        & \text{if } \alpha \in \bigcup_{\nu \in \kappa}I_{1, \nu}
\end{array}
\right.
\]

\noindent and it is clear that letting $\tau \in \mathfrak{U}^{\kappa}$ be constantly $\Av(x_0, x_1)$ we have $\mathcal{E}_1[\sigma]_{\sim} = [\tau]_{\sim}$.

\end{proof}

\end{section}

%
%
%

\end{document}